\documentclass[a4paper]{article}

\usepackage[utf8]{inputenc}
\usepackage{lmodern}
\usepackage{amssymb,amsfonts,amsmath,amsthm,bbm,mathrsfs,aliascnt,mathtools,pifont}
\usepackage{ulem}
\usepackage[english]{babel}
\usepackage[colorlinks]{hyperref}
\usepackage{tikz}
\usetikzlibrary{decorations}
\usetikzlibrary{decorations.pathreplacing}
\tikzset{individu/.style={draw,thick}}

\usepackage{pgfplots}

\theoremstyle{plain}
\newtheorem{theorem}{Theorem}[section]
\newtheorem{corollary}[theorem]{Corollary}
\newtheorem{lemma}[theorem]{Lemma}
\newtheorem{proposition}[theorem]{Proposition}

\theoremstyle{definition}
\newtheorem{definition}[theorem]{Definition}

\theoremstyle{remark}
\newtheorem{remark}[theorem]{Remark}
\newtheorem{example}[theorem]{Example}

\numberwithin{equation}{section}

\newcommand{\N}{\mathbb{N}}
\newcommand{\Z}{\mathbb{Z}}
\newcommand{\R}{\mathbb{R}}

\newcommand{\ind}[1]{\mathbf{1}_{\left\{#1\right\}}}
\newcommand{\indset}[1]{\mathbf{1}_{#1}}
\newcommand{\floor}[1]{{\left\lfloor #1 \right\rfloor}}

\newcommand{\crochet}[1]{{\langle #1 \rangle}}
\renewcommand{\bar}[1]{\overline{#1}}
\renewcommand{\tilde}[1]{\widetilde{#1}}
\renewcommand{\hat}[1]{\widehat{#1}}

\newcommand{\e}{\mathrm{e}}
\newcommand{\dd}{\mathrm{d}}

\DeclareMathOperator{\E}{\mathbb{E}}

\renewcommand{\P}{\mathbb{P}}

\renewcommand{\rho}{\varrho}
\renewcommand{\epsilon}{\varepsilon}

\title{Reinforced Galton--Watson processes III:\\ Empirical offspring distributions}
\author{Jean Bertoin\thanks{Institute of Mathematics, University of Zurich, Switzerland.} \and Bastien Mallein\thanks{Institut de Math\'ematiques de Toulouse,  Universit\'e de Toulouse, France.}}
\date{}
\begin{document}

\maketitle

\begin{abstract}
Reinforced Galton--Watson processes describe the dynamics of a population where reproduction events are reinforced, in the sense that offspring numbers of forebears can be repeated randomly by descendants. More specifically, the evolution depends on the empirical offspring distribution of each individual  along its ancestral lineage. We are interested here in asymptotic properties of the empirical distributions observed in the population, such as concentration, evanescence and persistence. For this, we incorporate tools from the theory of large deviations to our preceding analysis \cite{BM1, BM2}.
 \end{abstract}


\medskip


\section{Introduction}
\label{sec:introduction}

Given a locally finite rooted tree, which we think of as encoding the genealogical structure of a population, we look at out-degree sequences along branches from the root. For any vertex $v$ different from the root, this defines the \textit{empirical offspring distribution} at $v$ by
\begin{equation}
  \label{E:eod}
  \boldsymbol{\mu}_v\coloneqq \frac{1}{|v|} \sum_{i=0}^{|v|-1} \boldsymbol{\delta}_{d(v_i)},
\end{equation}
where $|v|\geq 1$ denotes the generation of $v$, $d(v_i)$ the number of children begotten by the forebear $v_i$ of $v$ at generation $i<|v|$, and,  as usual, $\boldsymbol{\delta}$ designates a Dirac measure.

The classical  Galton--Watson tree is an elementary  population model in which individuals reproduce independently and with a fixed reproduction law.
An easy combination of the well-known many-to-one formula for Galton--Watson trees and Sanov's theorem entails  that for any small enough neighborhood $G$ of some probability measure $\boldsymbol{\rho}$  on $\N=\{1,2, \ldots\}$, the mean number of vertices $v$ at generation $n$ with $\boldsymbol{\mu}_v\in G$ goes to $0$ or to $\infty$ exponentially fast  as $n\to \infty$, depending on whether the entropy of $\boldsymbol{\rho}$  relative to $\boldsymbol{\nu}$,
 \begin{equation}
  \label{E:entrop}
  H(\boldsymbol{\rho}| \boldsymbol{\nu})\coloneqq \sum \rho(k) \log\left(\frac{\rho(k)}{\nu(k)}\right),
\end{equation}
 is larger or smaller than
\begin{equation}
  \label{E:rhoell}
  \langle \boldsymbol{\rho}, \ln \rangle \coloneqq \sum \rho(k) \ln k.
\end{equation}
Here and throughout the rest of this text, the notation $\sum$ is used for summation over $k$, and we denote by $\ln$ the restriction of the natural logarithm $\log $ to $\N$.

In the first case, $\boldsymbol{\rho}$ is almost-surely evanescent for the empirical offspring distributions of the Galton--Watson tree, in the sense that with probability one, we can find a neighborhood $G$ of $\boldsymbol{\rho}$ such that the number of vertices $v$ with $\boldsymbol{\mu}_v\in G$ is finite. In the second case, $\boldsymbol{\rho}$ is persistent with positive probability, as a.s. on the event of survival of the tree, for any neighbourhood $G$ of $\boldsymbol{\rho}$, there will be infinitely many vertices $v$ such that $\boldsymbol{\mu}_v \in G$. Precise stronger statements will be given and proven in Section \ref{S:GWT}; they should be viewed as benchmarks for the present study.

\textit{Reinforced} Galton--Watson processes were introduced in \cite{BM1} by incorporating random repetitions of reproduction events in the  evolution of usual Galton--Watson processes. This depends on a memory parameter $q\in(0,1)$, which accounts for the frequency of repetitions.  For every $n\geq 1$, conditionally given the offspring numbers of individuals up to generation $n-1$, individuals at generation $n$ reproduce independently one from the others and as follows. With probability $q$, the number of children of $v$ is sampled from the empirical offspring distribution $\boldsymbol{\mu}_v$, and with complementary probability $1-q$, it is rather sampled from the reproduction law $\boldsymbol{\nu}$. In words, for every individual, with the obvious exception of the ancestor,  we  pick a forebear uniformly at random on its ancestral lineage. Then, either with probability $q$,  this individual begets the same number of children as the selected forebear, or with complementary probability $1-q$,  the number of its children  is  an independent sample from the reproduction law~$\boldsymbol{\nu}$.

{The growth of the averaged population size of reinforced Galton--Watson processes has been determined in \cite{BM1} using generating function techniques, and then the survival and a.s. growth rate have been further studied in \cite{BM2} introducing martingale techniques. Here, we incorporate tools from large deviations theory to our analysis and investigate asymptotic properties of empirical offspring distributions of individuals at large generations. Our main result, Theorem \ref{Tref}, states an exponential concentration around a specific reproduction law, and provides a sufficient condition for a given reproduction law to be evanescent or persistent; see Definition \ref{D1} below for a precise definition.}

The rest of the present work is organized as follow. Section~\ref{S:preli} is devoted to a few preliminaries on trees, basic background from large deviations theory, and the presentation of benchmark results about empirical offspring distributions in classical Galton-Watson processes. Section~\ref{S:Sanovreinforced} deals with versions of Sanov's Large Deviations Principle for reinforced sampling; it  mainly stems from earlier contributions by Budhiraja and Waterbury~\cite{BW} and ourselves~\cite{BM1}.
Our main results are presented and proven in Section \ref{S:Main}. {Last, in Section~\ref{S:survival}, we also point at a sufficient condition for survival of reinforced Galton--Watson processes, which generalizes than that obtained previously in \cite{BM2} and follows directly from our analysis.}

\section{Preliminaries}
\label{S:preli}

In this section, we first introduce some definition and properties related to the study of genealogical trees. In particular, we define the notions of \emph{evanescence} and \emph{persistence} for empirical offspring distributions. We then recall in Section~\ref{S:LDT} the Large Deviations Principles on finite alphabets, before applying it to the Galton-Watson tree in Section~\ref{S:GWT}.

\subsection{Some definition and a many-to-one formula}
\label{S:not}
We fix a probability measure $\boldsymbol{\nu}$ on  $\Z_+=\{0,1, \ldots\}$ with  support
$$S=\{k\in \Z_+:\nu(k)>0 \}.$$
We assume throughout this work that $S$ is finite\footnote{This restriction comes from our previous work \cite{BM1,BM2}, in which we observed that a reinforced Galton-Watson process with a reproduction law of unbounded support would grow at a super-exponential rate, thus would require very different methods to study.}, $\# S<\infty$;  the uninteresting case when $S\cap \N$ is a singleton is further ruled out. Recall that the restriction of the logarithmic function to $S$ is denoted here by $\ln$\footnote{ We found convenient to write $\ln k$ for the logarithm of a positive integer $k$ and $\log x$ when the argument is a positive real number $x$, even though of course the first function is merely the restriction of the second to $\Z_+$. That way, we treat $\ln$ as a vector in $\R^S$.}; it will play an important role for this study. This requires a few obvious conventions when $0\in S$, namely
\begin{equation} \label{E:conv}\ln 0 = -\infty\ ,\ \e^{-\infty}=0\ ,\ 0\ln 0=0.
\end{equation}

We write $\mathcal P_{S}\subset \R^{S}$ for the simplex of probability measures $\boldsymbol{\rho}$ on $S$; the support of such $\boldsymbol{\rho}$ may of course be a strict subset of $S$. We use the notation
$$\langle  \boldsymbol{\rho}, \lambda\rangle = \sum \rho(k) \lambda(k), \qquad \lambda\in [-\infty, \infty)^S,$$
where, by an earlier convention,  the sum runs over $k\in S$
(in particular, observe that $\langle  \boldsymbol{\rho}, \ln\rangle=-\infty$ iff $\rho(0)>0$).

Next, consider any rooted tree, say $T$, such that out-degrees of  vertices of $T$ always belong to $S$.  It is convenient to use implicitly the Ulam-Harris-Neveu framework\footnote{ In short, this amounts to enumerating progenies in $T$ and agreeing that the name of a child is obtained by aggregating its rank in the progeny at the end of the name of its parent; see e.g. \cite{Nev}.}, which enables us to  view $T$ as a subset of the Ulam tree $\bigcup_{n\geq 0} \N^n$.
We stress that $T$ is fixed (deterministic) in this section; however  the same notation $T$ will rather refer to a random tree from Section~\ref{S:GWT} onwards. Recall that empirical offspring distributions $\boldsymbol{\mu}_v$ of vertices of $T$ have been introduced in \eqref{E:eod}, so that $\boldsymbol{\mu}_v\in \mathcal P_{S}$. We now give a key definition for the present work.

\begin{definition} \label{D1} We say that $\boldsymbol{\rho}\in \mathcal P_{S}$ is:
\begin{enumerate}
\item[(i)] \textbf{evanescent} if there exists some neighborhood $G$ of $\boldsymbol{\rho}$ in $ \mathcal P_{S}$ such that
$$\#\{v\in T: \boldsymbol{\mu}_v\in G\}< \infty,$$

\item[(ii)] \textbf{weakly persistent} if $\boldsymbol{\rho}$ is not evanescent, that is, if
$$\#\{v\in T: \boldsymbol{\mu}_v\in G\}=\infty$$
 for any neighborhood $G$ of $\boldsymbol{\rho}$ in $ \mathcal P_{S}$,

\item[(iii)] \textbf{strongly persistent} if there exists an infinite line of descent $(v_n)_{n\geq 1}$ in $T$ with
$$\lim_{n\to \infty} \boldsymbol{\mu}_{v_n}=\boldsymbol{\rho}.$$
\end{enumerate}
\end{definition}
\begin{remark} \begin{enumerate}
\item The definition is only relevant for infinite trees, since by a compactness argument, $T$ is finite if and only if all probability laws on $S$ are evanescent.
\item Plainly, a law $\boldsymbol{\rho}$ with $\rho(0)>0$ cannot arise as an  empirical offspring distribution, and thus is necessarily evanescent.
\end{enumerate}
\end{remark}

\begin{example}
\label{expl}
We determine the persistent and evanescent laws for the tree $T$ constructed as follows: starting with a binary tree, we graft to each node of that tree an infinite line. In this tree, each individual has either 1 or 3 children, depending on whether they belong to the original binary tree or were later added. This tree can be written with Ulam-Harris notation as:
\[
  T = \left\{ u \in \cup_{n \in \Z_+} \{1,2,3\}^n : \text{ if } u(i) = 3 \text{ then } u(j) = 1 \text{ for all } i+1 \leq j \leq |u| \right\}.
\]
The strongly persistent measures are $\boldsymbol{\delta}_1$ and $\boldsymbol{\delta}_3$, whereas $t \boldsymbol{\delta}_1 + (1-t) \boldsymbol{\delta}_3$ is a weakly persistent measure for all $t \in (0,1)$. There are no evanescent law in $\mathcal{P}_{\{1,3\}}$ in that example.
\end{example}

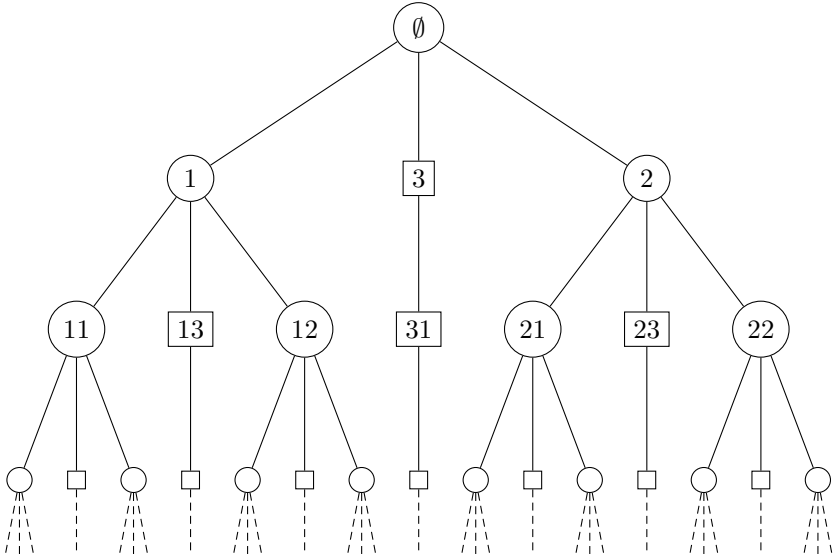
\begin{figure}[ht]
\begin{tikzpicture}[xscale=1.5]
      \node (n1) at (0,0) [circle, draw] {$\emptyset$};

      \node (n2) at (-2,-2) [circle, draw] {1};
      \node (n3) at (2,-2) [circle, draw] {2};
      \draw (n1) -- (n2);
      \draw (n1) -- (n3);

      \node (m1) at (0,-2) [rectangle, draw] {3};
      \draw (n1) -- (m1);
      \node (m11) at (0,-4) [rectangle, draw] {31};
      \draw (m1) -- (m11);
      \node (m111) at (0,-6) [rectangle, draw] {};
      \draw (m11) -- (m111);

      \node (n4) at (-3,-4) [circle, draw] {11};
      \node (n5) at (-1,-4) [circle, draw] {12};
      \node (n6) at (1,-4) [circle, draw] {21};
      \node (n7) at (3,-4) [circle, draw] {22};
      \draw (n2) -- (n4);
      \draw (n2) -- (n5);
      \draw (n3) -- (n6);
      \draw (n3) -- (n7);

      \node (m2) at (-2,-4) [rectangle, draw] {13};
      \draw (n2) -- (m2);
      \node (m21) at (-2,-6) [rectangle, draw] {};
      \draw (m2) -- (m21);

      \node (m3) at (2,-4) [rectangle, draw] {23};
      \draw (n3) -- (m3);
      \node (m31) at (2,-6) [rectangle, draw] {};
      \draw (m3) -- (m31);

      \node (n8) at (-3.5,-6) [circle, draw] {};
      \node (n9) at (-2.5,-6) [circle, draw] {};
      \node (n10) at (-1.5,-6) [circle, draw] {};
      \node (n11) at (-0.5,-6) [circle, draw] {};
      \node (n12) at (0.5,-6) [circle, draw] {};
      \node (n13) at (1.5,-6) [circle, draw] {};
      \node (n14) at (2.5,-6) [circle, draw] {};
      \node (n15) at (3.5,-6) [circle, draw] {};
      \draw (n4) -- (n8);
      \draw (n4) -- (n9);
      \draw (n5) -- (n10);
      \draw (n5) -- (n11);
      \draw (n6) -- (n12);
      \draw (n6) -- (n13);
      \draw (n7) -- (n14);
      \draw (n7) -- (n15);

      \node (m4) at (-3,-6) [rectangle,draw] {};
      \draw (n4) -- (m4);
      \node (m5) at (-1,-6) [rectangle,draw] {};
      \draw (n5) -- (m5);
      \node (m6) at (1,-6) [rectangle,draw] {};
      \draw (n6) -- (m6);
      \node (m7) at (3,-6) [rectangle,draw] {};
      \draw (n7) -- (m7);

      \draw [densely dashed] (n15) -- ++ (0,-1);
      \draw [densely dashed] (n15) -- ++ (.125,-1);
      \draw [densely dashed] (n15) -- ++ (-.125,-1);
      \draw [densely dashed] (n14) -- ++ (0,-1);
      \draw [densely dashed] (n14) -- ++ (.125,-1);
      \draw [densely dashed] (n14) -- ++ (-.125,-1);
      \draw [densely dashed] (n13) -- ++ (0,-1);
      \draw [densely dashed] (n13) -- ++ (.125,-1);
      \draw [densely dashed] (n13) -- ++ (-.125,-1);
      \draw [densely dashed] (n12) -- ++ (0,-1);
      \draw [densely dashed] (n12) -- ++ (.125,-1);
      \draw [densely dashed] (n12) -- ++ (-.125,-1);
      \draw [densely dashed] (n11) -- ++ (0,-1);
      \draw [densely dashed] (n11) -- ++ (.125,-1);
      \draw [densely dashed] (n11) -- ++ (-.125,-1);
      \draw [densely dashed] (n10) -- ++ (0,-1);
      \draw [densely dashed] (n10) -- ++ (.125,-1);
      \draw [densely dashed] (n10) -- ++ (-.125,-1);
      \draw [densely dashed] (n9) -- ++ (0,-1);
      \draw [densely dashed] (n9) -- ++ (.125,-1);
      \draw [densely dashed] (n9) -- ++ (-.125,-1);
      \draw [densely dashed] (n8) -- ++ (0,-1);
      \draw [densely dashed] (n8) -- ++ (.125,-1);
      \draw [densely dashed] (n8) -- ++ (-.125,-1);
      \draw [densely dashed] (m4) -- ++ (0,-1);
      \draw [densely dashed] (m5) -- ++ (0,-1);
      \draw [densely dashed] (m6) -- ++ (0,-1);
      \draw [densely dashed] (m7) -- ++ (0,-1);
      \draw [densely dashed] (m111) -- ++ (0,-1);
      \draw [densely dashed] (m21) -- ++ (0,-1);
      \draw [densely dashed] (m31) -- ++ (0,-1);

\end{tikzpicture}

\caption{A representation of the tree $T$ constructed in Example~\ref{expl} encoded with its Ulam--Harris notation. The vertices of the original binary tree are represented as circles.}
\end{figure}

In the forthcoming Section \ref{S:GWT}, we will analyze the case when $T$ is random and encodes a Galton--Watson process with reproduction law $\boldsymbol{\nu}$. We study the case of a multitype Galton--Watson tree in Section~\ref{sec:multitype}, generalizing the example above. The reinforced version, which is the main purpose of this work, will then be considered in Section \ref{S:Main}.

We end this section by presenting a version of the well-known many-to-one formula to compute the number of vertices at a fixed generation with a given empirical offspring distribution in $T$. Imagine that we distinguish  a (possibly finite)  line of descent $(U_n)$, randomly and recursively  as follows. The root of $T$ is of course the first element $U_0$ of the distinguished line. Then, for each $n\geq 0$, as long as the distinguished vertex $U_n$ is not a leaf of $T$, $U_{n+1}$  is chosen uniformly at random among the offspring of $U_n$. If $U_n$ is a leaf of $T$, then the distinguished line ends.

We shall refer to $(U_n)$ as a harmonic line\footnote{To avoid a possible misunderstanding, we stress that $U_n$ shall not be thought of as vertex picked uniformly at random amongst vertices at generation $n$, but rather results by selecting a child uniformly at random in a progeny, generation after generation. Observe that for the tree in Example~\ref{expl}, $U_2$ belongs to the binary subtree of $T$ with probability $4/9$, despite these vertices representing $4/7$th of the second generation. Indeed, at each step, the harmonic line has probability $2/3$ to stay in the binary subtree.} in $T$. For each generation $k$ for which the vertex $U_k$ is well-defined, we write $\hat{d}_k = d(U_k)$ for the number of children on the harmonic line at generation $k$. If the line has stopped by generation $k$, we write $\hat{d}_k = 0$ by convention. For all $n \in \N$, we denote by
\[
  \hat{\boldsymbol{\mu}}_{n} = \frac{1}{n} \sum_{j=0}^{n-1} \boldsymbol{\delta}_{\hat{d}_j},
\]
for the empirical offspring distribution at generation $n$ of the harmonic line. Note that $\hat{\boldsymbol{\mu}}_n = \boldsymbol{\mu}_{U_n}$ provided the line survives up to generation $n$. We also remark that $\hat{\boldsymbol{\mu}}(0) > 0$ if and only if the line stopped before generation $n$.

\begin{lemma}[Many-to-one formula]
\label{L1}
Take any  $(x_0,\ldots,x_{n-1}) \in (S\setminus\{0\})^n$ for some $n\geq 1$ and set $\boldsymbol{\mu} = \frac{1}{n} \sum_{j=0}^{n-1} \boldsymbol{\delta}_{x_j}$.
Using the notation \eqref{E:rhoell}, we have
\begin{align*}
&  \#\{v \in T : |v|=n \text{ and } d_{v_j} = x_j \text{ for all } j < n\}\\
&  = \P\left( \hat{d}_0 = x_0, \dots, \hat{d}_{n-1} = x_{n-1} \right) \times \exp(n \langle \boldsymbol{\mu}, \ln \rangle),
\end{align*}
In particular, for any probability measure $\boldsymbol{\rho}\in \mathcal P_S$ such that the measure $n\boldsymbol{\rho}$ is integer-valued, there is the identity
\begin{equation*}
  \#\{v\in T: |v|=n \text{ and } \boldsymbol{\mu}_v=\boldsymbol{\rho}\}
  = \P\left( \boldsymbol{\mu}_{U_n}=\boldsymbol{\rho}\right)  \times \exp(n \langle \boldsymbol{\rho}, \ln \rangle).
\end{equation*}
\end{lemma}
\begin{remark}
We stress that in the case $\rho(0)>0$,  $\boldsymbol{\rho}$ cannot be an  empirical offspring distribution, and that the right-hand side in the statement equals $0$  by our conventions.
\end{remark}

\begin{proof}
It is immediate from the definition of $U_n$ that for all $v \in T$ with $|v|=n$, we have
\[
  \P(U_n = v) = \P(U_1 = v_1,\ldots, U_n = v_n)= \prod_{j=0}^{n-1} \frac{1}{d(v_j)}.
\]
As a result, for any $n\geq 1$ and $(x_0, \ldots, x_{n-1})\in S^n$, there is the identity
\begin{align*}
  &\phantom{=}\P\left(d(U_0)=x_0, \ldots, d(U_{n-1})=x_{n-1}\right)\\
  &= \sum_{|v|=n} \P(U_n = v) \ind{d(v_0)=x_0,\ldots, d(v_{n-1})=x_{n-1}}\\
  &=\frac{1}{\prod_{j=0}^{n-1} x_j}\#\{v\in T: |v|=n \text{ and } d(v_i)=x_i \text{ for all }i=0, \ldots, n-1\}.
\end{align*}
The first claim follows readily rewriting
$\prod_{j=0}^{n-1} x_j = \exp(n\crochet{\boldsymbol{\mu},\ln}),$ the second one by summing over all  $(x_0,\ldots, x_{n-1})$ such that $\sum_{j=0}^{n-1} \boldsymbol{\delta}_{x_j} = n \boldsymbol{\rho}$.
\end{proof}

\subsection{Basic background from large deviations theory}
\label{S:LDT}
We briefly recall in this section a few key facts from large deviations theory for finite alphabets that will be needed later on, referring e.g. to \cite[Chapters~2 and~3]{DZ} and \cite{BD} for details and references.

We denote the logarithmic moment generating function of $\boldsymbol{\nu}$ by
$$\Lambda_0(\lambda)\coloneqq \log \left( \sum \nu(k)\e^{\lambda(k)}\right),\qquad \lambda\in \R^{S};$$
the index $0$ in $\Lambda_0$ is meant to indicate later on that we are dealing with the memory parameter $q=0$.
Its Fenchel-Legendre transform is defined as
$$\Lambda_0^*(\boldsymbol{\mu})\coloneqq \sup_{\lambda\in \R^{S}}\left( \langle \boldsymbol{\mu}, \lambda\rangle -\Lambda_0(\lambda)\right), \qquad \boldsymbol{\mu}\in \R^{S}.$$
Recall also from \eqref{E:entrop} that
$H(\boldsymbol{\rho}| \boldsymbol{\nu})$
stands for the entropy of  $\boldsymbol{\rho}$ relative to $\boldsymbol{\nu}$, and that the variational formula allows for the identification
\begin{equation} \label{E:var}
H(\boldsymbol{\rho}| \boldsymbol{\nu})=\Lambda_0^*(\boldsymbol{\rho}).
\end{equation}

The relative entropy of a measure is linked to large deviations theory by the fundamental theorem of Sanov \cite[Theorem 2.1.10]{DZ}. Given  some probability space  $(\Omega, \mathcal A,\P)$ and a sequence $(\xi_i)_{i\geq 1}$ of random variables in $S$, we  write
$$L_n\coloneqq \frac{1}{n} \sum_{i=1}^{n}\boldsymbol{\delta}_{\xi_i}, \qquad n\geq 1$$
for the sequence of empirical measures they induce.
Recall that  $(L_n)_{n\geq 1}$  satisfies the Large Deviation Principle (LDP) with rate function
$I:  \mathcal P_{S} \to [0,\infty]$ if the latter is a lower semi-continuous function, and if we have for every closed $F\subset \mathcal P_{S}$,
$$\limsup_{n\to \infty} \frac{1}{n} \log \P(L_n\in F) \leq -\inf_{\boldsymbol{\rho}\in F} I(\boldsymbol{\rho}),$$
and for every open $G\subset \mathcal P_{S}$,
$$\liminf_{n\to \infty} \frac{1}{n}  \log \P(L_n\in G) \geq -\inf_{\boldsymbol{\rho}\in G} I(\boldsymbol{\rho}).$$
In this setting, Sanov's Theorem states that if $\P_0$ is a probability measure under which the sequence $(\xi_i)$  is i.i.d. with law $\boldsymbol{\nu}$, then
the sequence of empirical measures $(L_n)_{n\geq 1}$  satisfies the LDP under $\P_0$ with  rate function
\begin{equation}
  \label{eqn:identifRateFuction}
  I_0=H(\cdot | \boldsymbol{\nu})=\Lambda_0^*.
\end{equation}

The Gibbs Conditioning Principle (see \cite[Section 3.3]{DZ}) is  an important consequence of Sanov's Theorem; it can be stated as follows. Recall that $(\xi_i)$  is i.i.d. with law $\boldsymbol{\nu}$ under $\P_0$, and consider any  closed convex subspace $\Gamma$ of $\mathcal P_{S}$ with a non-empty interior.
 For any $n\geq 1$, under the conditional probability measure $\P_0(\cdot \mid L_n\in \Gamma)$, the sequence $(\xi_i)_{1\leq i \leq n}$ is exchangeable, and if $\boldsymbol{\nu}_{\Gamma,n}$ denotes its one-dimensional marginal distribution,
 $$\nu_{\Gamma,n}(k)\coloneqq \P_0(\xi_1=k\mid L_n\in \Gamma), \qquad k\in S,$$
 then
 $$\lim_{n\to \infty} \boldsymbol{\nu}_{\Gamma,n} = \arg\min_{\Gamma} H(\cdot | \boldsymbol{\nu}),$$
 where the right-hand side stands for the unique\footnote{Using the compactness of $\Gamma$, and the convexity of $\Gamma$ and strict convexity of $H(\cdot|\boldsymbol{\nu})$, the existence of this unique minimizer is guaranteed.} $\boldsymbol{\gamma}_0\in \Gamma$ such that
 $$H(\boldsymbol{\gamma}_0 | \boldsymbol{\nu})= \min_{\Gamma} H(\cdot | \boldsymbol{\nu}).$$

 A related result in  this setting concerns the most likely trajectory taken by the process of empirical measures conditionally on the rare event $\{ L_n\in \Gamma\}$.
 For any sequence of positive integers $(k(n))_{n\geq 1}$ with
 $$\lim_{n\to \infty} k(n) = \infty \quad\text{and}\quad \limsup_{n\to \infty} k(n)/n\leq 1,$$
 there is the convergence in probability
 \begin{equation}\label{E:mostliketraj} \lim_{n\to \infty}  \P_0(\mathrm{dist}_{\mathcal P_S}(L_{k(n)},\boldsymbol{\gamma}_0)< \varepsilon \mid L_n\in \Gamma) =1, \qquad \text{for all } \varepsilon>0,
 \end{equation}
 where $\mathrm{dist}_{\mathcal P_S}$ stands, say, for  the maximal distance on $\mathcal P_S$. In other words, the most likely trajectory for $L_n$ to end up in $\Gamma$ consists in staying at all time within a ball of radius $\epsilon$ centred at $\boldsymbol{\gamma}_0$.
 We also refer to Mogulskii's Theorem \cite[Section~5.1]{DZ} for a much deeper sample path LDP for random walks.

\subsection{The benchmark case of Galton--Watson trees} \label{S:GWT}
This section focuses on the case when $T$ is a Galton--Watson tree with reproduction law $\boldsymbol{\nu}$, i.e. without reinforcement. We write
$\P_0$  for the distribution of $T$, where the index $0$ is again meant to stress that the memory parameter is $q=0$. We write also
$$m_{\boldsymbol{\nu}}\coloneqq \sum k \nu(k)$$
for the mean reproduction number,
 and  $\bar{\boldsymbol{\nu}}_0$ for the size-biased reproduction law, that is
$$\bar{\nu}_0(k)= k \nu(k)/m_{\boldsymbol{\nu}}, \qquad k\in S.$$

The following result will serve for us as a benchmark in the study of the reinforced setting; it analyzes the empirical offspring distribution at a large generation $n$ of $T$. Although it is well-known that a typical individual at the $n$th generation of a Galton-Watson process has an empirical offspring distribution close to the size-biased reproduction law $\bar{\boldsymbol{\nu}}_0$ (see e.g. \cite{Azais}), estimations of the number of individuals with a given empirical offspring distribution appear to be less well-known.

\begin{theorem} \label{T0} Recall the notation \eqref{E:rhoell}. The following assertions hold for usual Galton--Watson trees:
\begin{enumerate}
\item[(i)]\textbf{Concentration around  the size-biased reproduction law:} For every neighborhood $G$ of $\bar{\boldsymbol{\nu}}_0$, there exists $\varepsilon=\varepsilon(G)>0$ such that for all $n\geq 1$,
$$\E_0(\#\{v\in T: |v|=n \ \&\  \boldsymbol{\mu}_v\not \in G\})) \leq \e^{-\varepsilon n} \E_0(\#\{v\in T: |v|=n\}).$$
\item[(ii)] \textbf{Evanescent laws:} Any  ${\boldsymbol{\rho}}\in \mathcal P_S$ with
$$\langle \boldsymbol{\rho}, \ln \rangle < \Lambda_0^*(\boldsymbol{\rho})$$
 is evanescent, $\P_0$-a.s.

\item[(iii)] \textbf{Strongly persistent laws:} For any ${\boldsymbol{\rho}}\in \mathcal P_S$ with
$$\langle \boldsymbol{\rho}, \ln \rangle > H(\boldsymbol{\rho} |  \boldsymbol{\nu} ) , $$
$\boldsymbol{\rho}$ is strongly persistent, $\P_0$-a.s. conditionally on non-extinction.
\end{enumerate}
\end{theorem}
The quantities $\Lambda_0^*(\boldsymbol{\rho})$ and $H(\boldsymbol{\rho} |  \boldsymbol{\nu})$  in (ii) and (iii) are identical by the variational formula \eqref{E:var}; we used on purpose these two different naming conventions to make comparison with the reinforced case in the forthcoming Theorem~\ref{Tref} more transparent. We also observe that if
$\rho(0)>0$, then $\langle \boldsymbol{\rho}, \ln \rangle=-\infty$ and the condition in (ii) is automatically fulfilled (we already pointed out that such a probability measure is necessarily evanescent).

\begin{remark} \begin{enumerate}
\item We know that $\E_0(\#\{v\in T: |v|=n\})=m_{\boldsymbol{\nu}}^n.$ Further, in the critical or subcritical case $m_{\boldsymbol{\nu}}\leq 1$, $T$ is finite and any law ${\boldsymbol{\rho}}\in \mathcal P_S$ is obviously evanescent, $\P_0$-a.s. This does not contradict this claim of concentration around the size-biased reproduction law in expectation.

\item Theorem~\ref{T0} essentially shows that any distribution is either evanescent of strongly persistent for a Galton-Watson tree (with edge cases when $\langle \boldsymbol{\rho}, \ln \rangle = H(\boldsymbol{\rho} |  \boldsymbol{\nu} )$ to be separately treated). We consider in Section~\ref{sec:multitype} an example of a random tree with weakly --but not strongly-- persistent measures.
\item Given the concentration of the empirical offspring distribution of the population at time $n$ around $\bar{\boldsymbol{\nu}}_0$, it is worth noting that immediate algebra yield the identity
\begin{equation}
  \label{eqn:remark}
  \log m_{\boldsymbol{\nu}} - H(\boldsymbol{\rho}|\bar{\boldsymbol{\nu}}) = \crochet{\boldsymbol{\rho},\ln} - H(\boldsymbol{\rho}|\boldsymbol{\nu}).
\end{equation}
This gives an intuitive reason for the importance of comparing the quantities \eqref{E:entrop} with \eqref{E:rhoell} as mentioned in the introduction to estimate the presence of individuals with an empirical offspring distribution around $\boldsymbol{\rho}$.
\end{enumerate}
\end{remark}

Theorem \ref{T0}  can be seen as a disguised version of \cite[Theorem A]{Big78} that describe the distribution of individuals in a multidimensional branching random walk. To see the connection, let us remark that $$v \in T \mapsto (|v| {\mu}_v(j), j \in S)$$ is a branching random walk on $\Z^S$. In this branching random walk, an individuals at position $\boldsymbol{x} \in \Z^S$ creates $j$ children at position $\boldsymbol{x} + \boldsymbol{e_j}$ with probability ${\nu}(j)$, where $e_j(i) = \ind{i=j}$. Consequently, the set of non-empty vertices at the $n$th generation of this branching random walk can be well-approached by $nC$, where $C = \{\boldsymbol{x} \in \R^S : \Lambda^*(\boldsymbol{x}) < m_{\boldsymbol{\nu}} \}$. More precise estimates, such as the ones obtained in \cite[Theorem 5]{BiR05} for non-lattice branching random walks, could be adapted to the current setting as well.

The proof of Theorem \ref{T0} will easily follow from the well-known spinal decomposition which we now briefly present, tailored for our purposes.
Let $a: S\to \R_+$ be a nonnegative function of the out-degrees such that
\begin{equation}
  \label{E:har}
  \sum_{k\in S} ka(k) \nu(k)= 1.
\end{equation}
The value of $a(0)$ being irrelevant for \eqref{E:har}, we always assume, without loss of generality, that $a(0) = 0$.
We then denote by $\boldsymbol{\nu}^a$ the probability measure on $S$ given by
$$\nu^a(k)=k a(k) \nu(k).$$
The best known and most important example is when
$a(k) = 1/m_{\boldsymbol{\nu}}$ for all $k \geq 1$; then
$\boldsymbol{\nu}^a$ is the size-biased version $\bar{\boldsymbol{\nu}}_0$ of the reproduction law.

Recall that for every vertex $v\in T$, $d(v)$ denotes its out-degree, $|v|$ its generation, and $v_j$ its forebear at generation $j\leq |v|$.
It is seen from \eqref{E:har} and the branching property that the process
$$W^a_n\coloneqq \sum_{|v|=n} \prod_{j=0}^{n-1} a(d(v_j)) , \qquad n\geq 1,$$
is a $\P_0$-martingale with expectation $\E_0(W^a_n)=1$. In the case when
$a(k)\equiv 1/m_{\boldsymbol{\nu}}$,
$W^a_n$ is the fundamental martingale for Galton--Watson processes appearing e.g. in \cite[Theorem 1 on page 9]{AN}. This enables us to define
a probability measure $\P_0^a$ describing the joint law of a random tree $T$ together with an infinite random line of descent $(V^a_n)_{n\geq 0}$, usually referred to as a spine,
as follows. For any $n\geq 1$, any subtree $t_n$ of the Ulam tree with height $n$ and any $v$ vertex of $t_n$ at height $n$, if we write $T_n$ for the subtree obtained by pruning $T$ at generation $n$, then
\begin{equation}
  \label{E:defspin}
  \P^a_0\left(T_n= t_n, V^a_n=v\right)= \P_0(T_n=t_n) \prod_{j=0}^{n-1} a(d_{t_n}(v_j)),
\end{equation}
where $d_{t_n}(v_i)$ stands for the out-degree in $t_n$ of the forebear $v_i$ of $v$ at generation $i<n$.
\begin{remark}
In the case $\nu(0)=0$ when the probability of an empty progeny equals zero,
the harmonic line $(U_n)$ of Section \ref{S:not} is infinite and can also be viewed as a spine $(V^a_n)$ for the function  $a(k)\equiv k^{-1}/\sum j^{-1}\nu(j)$. The martingale is then trivial,  $W^a_n\equiv 1$.
\end{remark}

We can now state a necessary and sufficient condition for the uniform integrability of the martingale $(W_n^a)$, extending the Kesten-Stigum $L \log L$ criterion \cite{KestenStigum} for the case $a \equiv 1/m_{\boldsymbol{\nu}}$. This result can be seen as a version of Biggins' theorem \cite{Big77}, we adapt here a proof by Lyons~\cite{Lyo95} of this fact.

\begin{lemma} \label{L:spinal} The following assertions hold:
\begin{enumerate}
\item[(i)] Under $\P^a_0$, the sequence $(d(V^a_n))_{n\geq 0}$ of out-degrees along the spine
 is i.i.d. with law $\boldsymbol{\nu}^a$.

\item[(ii)] If
$$ \sum  \nu^a(k) \log a(k) <0,$$
then  $(W^a_n)_{n\geq 1}$ is a uniformly integrable  martingale under $\P_0$. Otherwise $\lim_{n \to \infty} W^a_n = 0$ a.s.
\end{enumerate}
\end{lemma}

\begin{proof} The first assertion is immediate from \eqref{E:defspin}. For the second, note from the first that under $\P^a_0$,
the process
$$\log\left( \prod_{j=0}^{n-1} a(d(V^a_j))\right) = \sum_{j=0}^{n-1} \log  a(d(V^a_j)), \qquad n\geq 1,$$
is a random walk with drift
$$\E_0^a\left( \log d(V^a_n)\right) =  \sum  \nu^a(k) \log a(k).$$

If $\sum \nu^a(k) \log a(k) < 0$, then $ \log\left( \prod_{j=0}^{n-1} a(d(V^a_j))\right)$ is a random walk with negative drift.
It follows from the law of large numbers that
$$\sum_{n=1}^{\infty} \prod_{j=0}^{n-1} a(d(V^a_j))< \infty, \qquad \P_0^a\text{-a.s.}$$
Since out-degrees are bounded, we deduce from the spinal decomposition (see \cite[Section 12]{BiK04}) that $\limsup_{n\to \infty} W^a_n<\infty$, $\P_0^a$-a.s.
By Durrett's criterion (see \cite[Theorem 4.3.5]{Dur19} or \cite[Theorem 3]{BiK04}), this entails in turns  the uniform integrability under $\P_0$.

Reciprocally, if $\sum \nu^a(k) \log a(k) \geq  0$, the drift of $ \log\left( \prod_{j=0}^{n-1} a(d(V^a_j))\right)$ is well-defined and non-negative. Thus
\[
  \limsup_{n \to \infty} W^a_n \geq  \limsup_{n \to \infty} \prod_{j=0}^{n-1} a(d(V^a_j)) = \infty \quad \P_0^a\text{-a.s.}
\]
Applying again Durrett's criterion, we deduce that $W^a_n \to 0$ a.s. under $\P_0$.
\end{proof}

We now have all the ingredients needed to establish Theorem \ref{T0}.

\begin{proof}[Proof of Theorem \ref{T0}]  (i) The average population size of a Galton--Watson process  at generation $n\geq 0$ is
 $$\E_0(\#\{v\in T: |v|=n\})= m_{\boldsymbol{\nu}}^n.$$
 Take $a(k)\equiv 1/m_{\boldsymbol{\nu}}$ and write $\bar{\P}_0=\P^a_0$.
 Since  then
 $$\prod_{j=0}^{n-1} a(d_{t_n}(v_j))= m_{\boldsymbol{\nu}}^{-n}$$ for any vertex $v_n$ at generation $n$ in any tree  $t_n$ with height $n$,
  we deduce from \eqref{E:defspin}  the identity
 $$\frac{ \E_0(\#\{v\in T: |v|=n \ \&\  \boldsymbol{\mu}_v \not \in G\})} { \E_0(\#\{v\in T: |v|=n\})} = \bar{\P}_0(L_n \in F),$$
 where $F=\mathcal P_S\backslash G$ and
 $$L_n\coloneqq \frac{1}{n} \sum_{i=1}^{n}\boldsymbol{\delta}_{d(V^a_{i-1})}.$$
 As we know from Lemma \ref{L:spinal}(i) that under $\bar{\P}_0$, $L_n$
 is the empirical measure of i.i.d. variables  distributed according to the size-biased reproduction law $\bar{\boldsymbol{\nu}}_0$, we deduce from Sanov's large deviations upper-bound that
 $$\limsup_{n\to \infty} \frac{1}{n} \log \bar{\P}_0(L_n \in F) \leq - \inf_{\boldsymbol{\rho}\in F} H(\boldsymbol{\rho} | \bar{\boldsymbol{\nu}}_0).$$
 The entropy relative to $\bar{\boldsymbol{\nu}}_0$,  $\boldsymbol{\rho}\mapsto  H(\boldsymbol{\rho} | \bar{\boldsymbol{\nu}}_0)$,  is a continuous map on $\mathcal{P}_S$ which is strictly positive except at $\bar{\boldsymbol{\nu}}_0$. Since $ F$ avoids a neighborhood of $ \bar{\boldsymbol{\nu}}_0$, the right-hand side above is strictly negative.

 (ii) Let $\boldsymbol{\rho}$ be as in the statement. By continuity of the entropy relative to $\boldsymbol{\nu}$, we can find  a closed
 neighborhood $F$ of  ${\boldsymbol{\rho}}$ such that
 $$\sum \rho'(k) \log k < H(\boldsymbol{\rho}' |  \boldsymbol{\nu}), \qquad\text{ for all }\boldsymbol{\rho}' \in F.$$
 On the other hand, we have seen in (i) that
 $$\E_0(\#\{v\in T: |v|=n \ \&\  \boldsymbol{\mu}_v  \in F\}) = m_{\boldsymbol{\nu}}^n \times \bar{\P}_0(L_n \in F);$$
 and then again by Sanov's large deviations upper-bound,
 $$\limsup_{n\to \infty} \frac{1}{n} \log \E_0(\#\{v\in T: |v|=n \ \&\  \boldsymbol{\mu}_v  \in F\}) \leq
\log (m_{\boldsymbol{\nu}})- \inf_{\boldsymbol{\rho}'\in F} H(\boldsymbol{\rho}' | \bar{\boldsymbol{\nu}}_0).$$

It now suffices to write
$$ H(\boldsymbol{\rho}' | \bar{\boldsymbol{\nu}}_0)= \sum \rho'(k) \log\left(\frac{ \rho'(k) m_{\boldsymbol{\nu}}}{k \nu(k)}\right)= H(\boldsymbol{\rho}' | {\boldsymbol{\nu}}) + \log(m_{\boldsymbol{\nu}})- \sum \rho'(k) \ln k,$$
see \eqref{eqn:remark}, to get
$$\limsup_{n\to \infty} \frac{1}{n} \log \E_0(\#\{v\in T: |v|=n \ \&\  \boldsymbol{\mu}_v  \in F\}) <0.$$
Summing over generations yields
$$\E_0(\#\{v\in T:  \boldsymbol{\mu}_v  \in F\}) <\infty,$$
and we conclude that $\boldsymbol{\rho}$ is evanescent, $\P_0$-a.s.

(iii) Let $\boldsymbol{\rho}$ be as in the statement; in particular $\rho(0)=0$. We set $a(k)\coloneqq \rho(k)/(k\nu(k))$ for $k\in S$ --which satisfies \eqref{E:har}-- so that $\boldsymbol{\rho}= \boldsymbol{\nu}^a$.
By Lemma \ref{L:spinal}(i) and the law of large numbers for the sequence $(d(V^a_n))_{n\geq 0}$ of out-degrees along the spine, we see that $\boldsymbol{\rho}$ is strongly persistent $\P_0^a$-a.s.
On the other hand, we have by the variational formula that
$$
  \sum \nu^a(k) \ln k = \crochet{\boldsymbol{\rho},\ln} > H(\boldsymbol{\rho}|\boldsymbol{\nu}) = \sup_{\lambda \in \R^S} \left( \crochet{\boldsymbol{\rho},\lambda} - \Lambda_0(\lambda)\right).
$$
Applying this inequality to $\lambda(k) = \log (k a(k))$, so that  $\Lambda_0(\lambda) = 0$, we have
$$ \sum \nu^a(k) \ln k > \sum \nu^a(k) \log (k a(k)),$$
so the assumption of Lemma \ref{L:spinal}(ii) holds.

Note from \eqref{E:defspin} that for every $n\geq 1$, the distribution of $T_n$ under $\P^a_0$ is absolutely continuous with respect to that under $\P_0$ with density $W^a_n$, so the uniform integrability of the martingale $W^a_n$ enables us to deduce
that the distribution of whole tree $T$ under $\P^a_0$ is again absolutely continuous with respect to that under $\P_0$, with density given by the terminal value $W^a_{\infty}$.
Hence  $\boldsymbol{\rho}$ is strongly persistent with a strictly positive probability under $\P_0$, and the stronger claim conditionally on non-extinction follows from a standard argument involving the branching property.
\end{proof}

We now conclude this section with a refinement of Theorem \ref{T0} which will be useful in the next section. This result allows to estimate the almost sure growth rate of the number of individuals with an empirical distribution in a neighbourhood of $\boldsymbol{\rho}$.
\begin{lemma} \label{lem}
Let $\boldsymbol{\rho} \in \mathcal{P}_S$ and $G$ an open neighborhood  of $\boldsymbol{\rho}$.
\begin{enumerate}
  \item[(i)] If $m_{\boldsymbol{\nu}} >1$ and  $\crochet{\boldsymbol{\rho},\ln} - H(\boldsymbol{\rho}|\boldsymbol{\nu}) >0$, then we have     \[
    \liminf_{n \to \infty} \frac{1}{n} \log \#\{v\in T: |v|=n\, \& \, \boldsymbol{\mu}_v \in G\} \geq \crochet{\boldsymbol{\rho},\ln} - H(\boldsymbol{\rho}|\boldsymbol{\nu}), \]
    $\P_0$-a.s. on the survival event (i.e. when $T$ is infinite).
  \item[(ii)] If $m_{\boldsymbol{\nu}} <1$, then we have
  \[
    \liminf_{n \to \infty} \frac{1}{n} \log \P_0(\exists v\in T: |v|=n \,\&\, \boldsymbol{\mu}_v \in G) \geq  \crochet{\boldsymbol{\rho},\ln} - H(\boldsymbol{\rho}|\boldsymbol{\nu}).
  \]
\end{enumerate}
\end{lemma}

\begin{proof}
(i) We prove this result using an analogue reasoning to \cite{Big77b}. Let $B$ be an open (convex) ball centered at $\boldsymbol{\rho}$ with radius $r>0$ sufficiently small so that the ball of same center and  radius $2r$ is contained in $G$. We recall that, by Lemma~\ref{L:spinal} and Sanov's large deviations theorem, we have
\[
  \liminf_{n \to \infty} \frac{1}{n} \log \E_0(\#\{v\in T: |v|=n\, \& \,  \boldsymbol{\mu}_v \in \bar{G}\}) \geq \crochet{\boldsymbol{\rho},\ln} - H(\boldsymbol{\rho}|\boldsymbol{\nu}).
\]
Using that this quantity is positive by our assumptions, we construct a supercritical Galton-Watson tree embedded in $T$ as follows. For all $\epsilon >0$ small enough, we choose $K$ such that
\begin{equation}
  \label{eqnK}
  \frac{1}{K} \log \E_0(\#\{v\in T: |v|=K \, \& \,  \boldsymbol{\mu}_v \in B\}) \geq \crochet{\boldsymbol{\rho},\ln} - H(\boldsymbol{\rho}|\boldsymbol{\nu}) - \epsilon >0
\end{equation}
The first generation of that Galton-Watson tree $T^{(K)}$ is given by $\{ |v|=K : \boldsymbol{\mu}_v \in B \}$, the second generation by the descendant of those with an empirical offspring distribution between times $K$ and $2K$ that belong to $G$, and so on. By \eqref{eqnK}, this Galton-Watson tree is supercritical, with
\[
  \liminf_{p \to \infty} \frac{1}{pK} \#\{v \in T^{(K)} : |v|=p\} \geq \crochet{\boldsymbol{\rho},\ln} - H(\boldsymbol{\rho}|\boldsymbol{\nu}) - \epsilon ,
\]
a.s. on the survival of $T^{(K)}$. Moreover, by convexity of $B$, all individuals $v \in T^{(K)}$ verify that $\boldsymbol{\mu}_v \in B$. Additionally, for all $p>0$ large enough, we have $\boldsymbol{\mu}_{v_p} \in G$. We conclude that
\[
  \P_0\left( \liminf_{n \to \infty} \frac{1}{n} \log \#\{v\in T: |v|=n  \, \& \, \boldsymbol{\mu}_v \in G\} \geq \crochet{\boldsymbol{\rho},\ln} - H(\boldsymbol{\rho}|\boldsymbol{\nu}) \right) >0.
\]
We then use the $0$-$1$ law for Galton-Watson processes, since our event of interest is hereditary, to conclude.

(ii) Without loss of generality again, we assume here that $G$ is an open ball centered at $\boldsymbol{\rho}$. The second moment method yields
\[
  \P_0(\exists v\in T: |v|=n\, \& \,  \boldsymbol{\mu}_v \in G) \geq  \frac{\E_0\left( \#\{v\in T: |v|=n\, \& \,  \boldsymbol{\mu}_v \in G\}\right)^2}{\E_0\left( \#\{v\in T: |v|=n\, \& \, \boldsymbol{\mu}_v \in G\}^2\right)}.
\]
By Lemma~\ref{L:spinal}, we recall that
\begin{align*}
&  \liminf_{n \to \infty} \frac{1}{n} \log \E_0\left( \#\{v\in T: |v|=n\, \& \,  \boldsymbol{\mu}_v \in G\}\right)\\
  &= \log m_{\boldsymbol{\nu}} + \liminf_{n \to \infty} \frac{1}{n} \log \bar{\P}_0(L_n \in G)\\
  &\geq - \inf_{\boldsymbol{\mu} \in G} (H(\boldsymbol{\mu}|\boldsymbol{\nu})  - \crochet{\mu,\ln}),
\end{align*}
with the same notation and computations as in the proof of Theorem~\ref{T0}(i).

On the other hand,  we have
\begin{align*}
 &\phantom{=}\E_0\left( \#\{v\in T: |v|=n\, \& \,  \boldsymbol{\mu}_v \in G\}^2\right)\\
  &= m_{\boldsymbol{\nu}}^n \bar{\E}_0 \left( \#\{v\in T: |v|=n\, \& \, \boldsymbol{\mu}_v \in G\} \ind{L_n \in G}\right)\\
  &=  m_{\boldsymbol{\nu}}^n\sum_{k=0}^{n-1} \bar{\E}_0 \left( \#\{v\in T: |v|=n\, \& \,  \boldsymbol{\mu}_v \in G, |v \wedge U_n| = k\} \ind{L_n \in G} \right)\\
  &\ \ +m_{\boldsymbol{\nu}}^n\P_0(L_n \in G),
\end{align*}
where $|v \wedge U_n|$ represents the time at which $v$ splits from the harmonic line $U_n$. We write $M=\sup S$ for the largest possible number of children of an individual, and  remark that for all $k \leq n$, we have
\[
  \bar{\E}_0 \left( \#\{v\in T: |v|=n\, \& \, \boldsymbol{\mu}_v \in G, |v \wedge U_n| = k\} \ind{L_n \in G} \right) \leq M m_{\boldsymbol{\nu}}^{n-k} \bar{\P}_0(L_n \in G).
\]
Therefore
\[
  \E_0\left( \#\{v\in T: |v|=n\, \& \,  \boldsymbol{\mu}_v \in G\}^2\right) \leq (n+1) M m_{\boldsymbol{\nu}}^n \P_0(L_n \in G).
\]
As a result, another call to Sanov's large deviations theorem yields
\[
  \limsup_{n \to \infty} \frac{1}{n} \log  \E_0\left( \#\{v\in T: |v|=n\, \& \, \boldsymbol{\mu}_v \in G\}^2\right) \leq - \inf_{\boldsymbol{\mu} \in \bar{G}} (H(\boldsymbol{\mu}|\boldsymbol{\nu}) - \crochet{\boldsymbol{\mu},\ln}).
\]

We conclude by continuity of $\boldsymbol{\mu} \mapsto H(\boldsymbol{\mu}|\boldsymbol{\nu}) - \crochet{\boldsymbol{\mu},\ln}$ that
\begin{align*}
  \liminf_{n \to \infty} \frac{1}{n} \log  \P_0(\exists v\in T: |v|=n\, \& \,  \boldsymbol{\mu}_v \in G)
  &\geq - \inf_{\boldsymbol{\mu} \in G} H(\boldsymbol{\mu}|\boldsymbol{\nu}) - \crochet{\boldsymbol{\mu},\ln}\\
  &\geq \crochet{\boldsymbol{\rho},\ln} - H(\boldsymbol{\rho}|\boldsymbol{\nu}). \qedhere
\end{align*}
\end{proof}

\subsection{Weak persistence in a two-type Galton-Watson tree}
\label{sec:multitype}
The purpose of this section simply to exhibit natural examples of random trees in which there exist weakly but not strongly persistent laws for the empirical offspring distributions. We stress that we do not aim to obtain either a general result or optimal conditions.

Let $\boldsymbol{\nu}$, $\boldsymbol{\nu}'$ be two probability distributions on $\Z_+$ with finite support and such that
\begin{equation}
  \label{eqn:simplifyingAssumption}
  m_{\boldsymbol{\nu}} = \sum k \nu(k)  > 1 > m_{\boldsymbol{\nu}'} = \sum k \nu'(k).
\end{equation}
So a $\boldsymbol{\nu}$-Galton-Watson tree is infinite with positive probability, while a $\boldsymbol{\nu}'$-Galton-Watson tree is a.s. finite.
We consider a two-type Galton-Watson process, where
 each individual has a type in $\{1,2\}$, and given its type, reproduces independently from the rest of the population. We suppose that individuals of type $1$ reproduce by creating a random number of children of type $1$ with law $\boldsymbol{\nu}$, as well as one child of type $2$. Particles of type $2$ only create children of type $2$, making a random number with law $\boldsymbol{\nu}'$.

Remarking that the total number of children of each individual of type $1$ is given by $k+1$ for some $k$ such that $\nu(k) >0$, we write
\[
  \bar{S} = \{k \geq 2 : \nu(k-1)> 0\} \cup \{k \geq 1 : \nu'(k)>0\},
\]
with $S$ the support of $\boldsymbol{\nu}$. We also define the operator $\tau$ associating to a probability measure $\boldsymbol{\rho}$ its pushforward by $k \mapsto k-1$. We prove in this section the following result.
\begin{proposition}
\label{prop}
With the two-types tree $T$ defined above, a law $\boldsymbol{\rho} \in \mathcal{P}_{\bar{S}}$ is
\begin{enumerate}
  \item[(i)] \textbf{strongly persistent} if $\crochet{\tau \boldsymbol{\rho},\ln} > H(\tau \boldsymbol{\rho}|\boldsymbol{\nu})$;
  \item[(ii)] \textbf{weakly persistent} if there exists $s \in (0,1)$ and $\boldsymbol{\mu}, \boldsymbol{\mu}'$ with
 $$
    \boldsymbol{\rho} = s \boldsymbol{\mu} + (1-s) \boldsymbol{\mu}',$$
  such that the following two inequalities hold:
$$ \crochet{\tau \boldsymbol{\mu},\ln}> H(\tau \boldsymbol{\mu}|\boldsymbol{\nu})$$
and
$$
  s\crochet{\tau \boldsymbol{\mu},\ln} + (1-s) \crochet{\boldsymbol{\mu}',\ln}  >  s H(\tau \boldsymbol{\mu}|\boldsymbol{\nu}) + (1-s) H(\boldsymbol{\mu}'|\boldsymbol{\nu}').
$$
\end{enumerate}
\end{proposition}

Using this proposition, one can construct a variety of random trees with weakly --but not strongly-- persistent offspring distributions. For example, if there exists $k \in \N$ such that $\nu'(k)>0$ but $\nu(k-1) = 0$, then any individual $v$ such that $\boldsymbol{\mu}_v(k) >0$ must be of type $2$. However, by \eqref{eqn:simplifyingAssumption}, almost surely there is no infinite line of individuals of type $2$. Consequently, any law $\boldsymbol{\rho}$ such that $\rho(k)>0$ cannot be strongly persistent.

\begin{proof}[Proof of Proposition~\ref{prop}]
(i)
This is an immediate consequence of Theorem~\ref{T0}, using that the subtree $T^{(1)}$ of $T$, consisting of only individuals of type 1 is a Galton-Watson tree with reproduction law $\boldsymbol{\nu}$, and the empirical offspring distribution of $v \in T^{(1)}$ verifies $\boldsymbol{\mu}^{(1)}_v = \tau \boldsymbol{\mu}_v$, with $\boldsymbol{\mu}^{(1)}_v$ representing the empirical offspring distribution of $v$ as an element of $T^{(1)}$.

(ii)
Let $\boldsymbol{\rho}$, $\boldsymbol{\mu}$, $\boldsymbol{\mu}'$ and $s$ as in the statement. By continuity of $H(\cdot|\boldsymbol{\nu})$ and $H(\cdot|\boldsymbol{\nu'})$, we observe that there exist open neighborhoods $O$, $G$, $G'$ of $\boldsymbol{\rho}$, $\boldsymbol{\mu}$, $\boldsymbol{\mu}'$ respectively, such that for all $(\boldsymbol{\xi}, \boldsymbol{\xi}') \in G \times G'$, $s \boldsymbol{\xi} + (1-s) \boldsymbol{\xi}' \in O$. We fix $\epsilon > 0$ such that
\begin{align*}
 &\crochet{\tau \boldsymbol{\mu},\ln} - H(\tau \boldsymbol{\mu}|\boldsymbol{\nu}) - \epsilon > 0\\
 \text{ and } &s\crochet{\tau \boldsymbol{\mu},\ln} + (1-s) \crochet{\boldsymbol{\mu}',\ln}  - \left( s H(\tau \boldsymbol{\mu}|\boldsymbol{\nu}) + (1-s) H(\boldsymbol{\mu}'|\boldsymbol{\nu}')\right) - 2\epsilon >0.
\end{align*}

By Lemma~\ref{lem}(i), we observe that almost surely on the survival event, for all $n$ large enough, there will be at least
\[
  \exp\left( s n \left(\crochet{\tau \boldsymbol{\mu},\ln} - H(\tau \boldsymbol{\mu}|\boldsymbol{\nu}) - \epsilon \right) \right)
\]
individuals $v$ of type $1$ at generation $\floor{ns}$ with empirical offspring distribution $\boldsymbol{\mu}_v \in G$. In addition, by Lemma~\ref{lem}(ii), each of these individuals will give birth to one individual of type $2$, that has probability at least
\[
  \exp\left( (1-s) n \left(\crochet{\boldsymbol{\mu}',\ln} - H(\boldsymbol{\mu}'|\boldsymbol{\nu}') - \epsilon \right)  \right)
\]
to start a Galton-Watson tree of reproduction law $\boldsymbol{\nu}'$ creating at least one descendant at generation $n - \floor{ns} -1$ with an empirical offspring distribution in $G'$. Then, using Borel-Cantelli lemma, we conclude that almost surely on the survival event, for all $n$ large enough, there will be an individual $|v|=n$ in $T$ with an empirical offspring distribution in $O$, completing the proof.
\end{proof}

\section{Reinforced sampling} \label{S:Sanovreinforced}
We return here to the setting of Section \ref{S:LDT}, where the notation $\P_0$  was used for a probability measure under which the $\xi_i$ are i.i.d. with law  $\boldsymbol{\nu}$. The dynamics of reinforced Galton--Watson processes that have been depicted in the Introduction incite us to define another probability measure  denoted by $\P_q$, where $q\in(0,1)$ is the memory parameter, under which the sequence $(\xi_i)$ rather describes balls added successively to an urn in a P\'olya type  process.

Namely, we view $S$ as a set of colors and imagine an urn containing colored balls. In some probability space $(\Omega, \mathcal A, \P_q)$,  we pick a ball uniformly at random in the urn, observe its color and return it to the urn together with a new ball. The color of the new ball is the same as that just sampled with probability $q$, and with complementary probability $1-q$, its color is rather chosen randomly according to $\boldsymbol{\nu}$. We then iterate independently,  supposing for definitiveness that at time $n=1$, we add to an initially empty urn a single ball with random color $\xi_1$ distributed according to $\boldsymbol{\nu}$. We write $\xi_n$ for the color of the  $n$-th ball added to the urn. It should be plain that under $\P_q$,  each variable $\xi_i$ has again the law  $\boldsymbol{\nu}$; however the sequence with memory $(\xi_i)_{i\geq 1}$ is not i.i.d., nor Markovian, nor even stationary.

Just as in the previous section, we denote   by $(L_n)_{n\geq 1}$ the sequence of empirical measures induced by the sequence of colors $(\xi_i)_{i\geq 1}$. Perhaps the most basic result   in the reinforced setting is that the Law of Large Numbers remains valid, namely, $L_n$ converges to $\boldsymbol{\nu}$ as $n\to \infty$,  $\P_q$-a.s. This is indeed a special case of a much more general result for random urn schemes, see e.g. \cite{AK} or \cite{Jan04}. The main purpose of this section is to investigate the effect of reinforcement on Sanov's Theorem and some of its consequences.
On our way, we briefly study a function that has a crucial role in the analysis.

 \subsection{Reinforced  Sanov's theorem}

Budhiraja and Waterbury \cite{BW} have established the LDP for a more general family of reinforced chains on a finite alphabet, in which the rate function is given in the form of an optimization problem.
In the present setting, their main result  can be stated as follows.

\begin{theorem}[After Budhiraja and Waterbury \cite{BW}] \label{T:LDPBW}
 The sequence of empirical measures $(L_n)$ satisfies the LDP  under $\P_q$ with  rate function
$I_q: \mathcal P_{S}\to \R_+$ given by
\[
  I_q(\boldsymbol{\rho}) = \inf_{\boldsymbol{\eta} \in {\mathcal{U}}(\boldsymbol{\rho})} \int_0^1 H(\boldsymbol{\eta}_s | q \boldsymbol{\psi}_s + (1-q)\boldsymbol{\nu}) \dd s,
\]
where for all $t \in (0,1]$
\begin{equation} \label{E:vareil}
\boldsymbol{\psi}_t = \frac{1}{t} \int_0^t \boldsymbol{\eta}_s \dd s,
\end{equation}
and ${\mathcal U} (\boldsymbol{\rho})$ denotes the collection of controls $\boldsymbol{\eta}=(\boldsymbol{\eta}_s)_{s \in [0,1]}$ in $\mathcal P_{S}$ such that $\boldsymbol{\psi}_1 = \int_0^1 \boldsymbol{\eta}_s \dd s = \boldsymbol{\rho}$.
\end{theorem}

\begin{proof} Specializing \cite[Theorem 2]{BW} gives the LDP for $(L_n)$ with  rate function
$$I_q(\boldsymbol{\rho}) \coloneqq \inf_{\boldsymbol{\eta} \in \mathcal V (\boldsymbol{\rho})} \int_0^{\infty} \e^{-s} H( \boldsymbol{\eta}_s | q \boldsymbol{\phi}_s + (1-q) \boldsymbol{\nu}) \dd s,
$$
where
\begin{equation} \label{E:eil}
\boldsymbol{\phi}_t = \boldsymbol{\rho}-\int_0^t \boldsymbol{\eta}_s \dd s + \int_0^t \boldsymbol{\phi}_s \dd s,
\end{equation}
and
$ \mathcal V (\boldsymbol{\rho})$ denotes the collection of controls $\boldsymbol{\eta}=(\boldsymbol{\eta}_s)_{s\geq 0}$ in $\mathcal P_{S}$ such that  $\boldsymbol{\phi}_t\in \mathcal P_{S}$  for all $t\geq 0$.
We then perform the simple change of variables $s \mapsto 1-\e^{-s}$ and get
\begin{align*}
 & \int_0^{\infty} \e^{-s} H( \boldsymbol{\eta}_s | q \boldsymbol{\phi}_s + (1-q) \boldsymbol{\nu}) \dd s\\
 & = \int_0^1 H(\boldsymbol{\eta}_{-\log (1-s)} | q \boldsymbol{\phi}_{-\log(1-s)} + (1-q) \boldsymbol{\nu}) \dd s.
\end{align*}
On the other hand, we remark that if $\boldsymbol{\eta} \in \mathcal{V}(\boldsymbol{\rho})$, then the solution to \eqref{E:eil} is
\[
  \boldsymbol{\phi}_{-\log(1-t)} = \frac{1}{1-t}\left( \boldsymbol{\rho}-\int_0^t \boldsymbol{\eta}_{-\log(1-s)} \dd s\right) = \frac{1}{1-t} \int_t^1 \boldsymbol{\eta}_{-\log (1-s)} \dd s,
\]
where for the second equality, we used that
$$\lim_{t \to \infty} \e^{-t} \boldsymbol{\phi}_t = 0 = \boldsymbol{\rho} - \int_0^\infty e^{-s} \boldsymbol{\eta}_s \dd s.$$ Then, writing $\bar{\boldsymbol{\eta}}_t = \boldsymbol{\eta}_{-\log t}$ and $\boldsymbol{\psi}_t = \boldsymbol{\phi}_{-\log t}$, we obtain the alternative formula for $I_q$ of the statement, remarking that $\mathcal{P}(S)$ is convex, therefore for any possible control, $\frac{1}{t} \int_0^t \bar{\boldsymbol{\eta}}_s \dd s \in \mathcal{P}(S)$ .
\end{proof}

Solving explicitly the optimization problem in Theorem \ref{T:LDPBW} to get an explicit expression for the rate function $I_q$ does not seem easy, and
the purpose of the next section is to offer an alternative characterization which will be much simpler to analyze.  Before this, we point at the following upper-bound
which may be quite sharp, at least in simple cases, as Example~\ref{E:es} below suggests.

\begin{corollary} \label{C:IBW}
For every $\boldsymbol{\rho} \in \mathcal P_{S}$, there is  the  upper-bound
\begin{equation}\label{E:IBW}
I_q(\boldsymbol{\rho})\leq H(\boldsymbol{\rho} | q\boldsymbol{\rho} + (1-q) \boldsymbol{\nu}).
\end{equation}
Moreover, if $\boldsymbol{\rho}$ is equivalent to $\boldsymbol{\nu}$, that is, $\rho(k)>0$ for all $k\in S$, and if  $\boldsymbol{\rho}\neq \boldsymbol{\nu}$,  then the inequality \eqref{E:IBW} is strict.
\end{corollary}

\begin{proof} The first claim is immediate from Theorem \ref{T:LDPBW}.
Indeed, for the constant control $\boldsymbol{\eta}_s\equiv  \boldsymbol{\rho}$, we have $\boldsymbol{\psi}_t \equiv  \boldsymbol{\rho}$, and \eqref{E:IBW} follows.

Suppose now that $\boldsymbol{\rho}\neq \boldsymbol{\nu}$ and that $\rho(k)>0$ for all $k\in S$. Set
$$\boldsymbol{\rho}_x\coloneqq \boldsymbol{\rho}+ x (\boldsymbol{\rho}-\boldsymbol{\nu})\quad \text{ for }x\in \R,$$ so $\boldsymbol{\rho}_x \in \mathcal P_{S}$ provided that $|x|$ is small enough. The function
$$h(x,y)\coloneqq H(\boldsymbol{\rho}_x\mid q \boldsymbol{\rho}_y+(1-q)\boldsymbol{\nu})$$
is $\mathcal{C}^{\infty}$ on some neighborhood of $(0,0)$, and we have
\begin{equation} \label{E:Taylor}
h(x,y)=h(0,0)+ x g_1+y g_2 + O(x^2+y^2),
\end{equation}
where $(g_1, g_2)= \nabla h(0,0)$.

The key observation is that
\begin{equation}\label{E:snpd}
g_2<0.
\end{equation}
Indeed, for $y>0$ sufficiently small, we can express $q \boldsymbol{\rho}_y+(1-q)\boldsymbol{\nu}$ as a convex combination of $q \boldsymbol{\rho}+(1-q)\boldsymbol{\nu}$ and
$\boldsymbol{\rho}$, namely
$$q \boldsymbol{\rho}_y+(1-q)\boldsymbol{\nu}= \left(1-\frac{q y}{1-q}\right) \left( q \boldsymbol{\rho}+(1-q)\boldsymbol{\nu}\right) + \frac{q y}{1-q} \boldsymbol{\rho}.$$
By convexity of the function $-\log$, this yields
$$h(0,y) \leq  \left(1-\frac{q y}{1-q}\right)h(0,0)+ \frac{q y}{1-q}H(\boldsymbol{\rho}\mid \boldsymbol{\rho})=  \left(1-\frac{q y}{1-q}\right)h(0,0),$$
and then
$$\partial_yh(0,0)\leq -\frac{q}{1-q}h(0,0).$$
We have $h(0,0)>0$ as $\boldsymbol{\rho}\neq q \boldsymbol{\rho}+(1-q)\boldsymbol{\nu}$, and  \eqref{E:snpd} holds.

We now return to the setting of Theorem \ref{T:LDPBW} and consider the control $\boldsymbol{\eta}$ given by
$$\boldsymbol{\eta}_t= \left\{ \begin{matrix} \boldsymbol{\rho}_\varepsilon & \text{ if }0\leq t \leq 1/2,\\
\boldsymbol{\rho}_{-\varepsilon} & \text{ if }1/2 < t \leq 1,
\end{matrix} \right.$$
for some sufficiently small $\varepsilon>0$.
We get
$$\boldsymbol{\psi}_t = \left\{ \begin{matrix} \boldsymbol{\rho}_\varepsilon & \text{ if }0\leq t \leq 1/2,\\
\boldsymbol{\rho}_{\varepsilon(1/t-1)} & \text{ if }1/2 < t \leq 1.
\end{matrix} \right.$$
By the Taylor expansion \eqref{E:Taylor}, we deduce
\begin{align*}&
\int_0^1 H(\boldsymbol{\eta}_s | q \boldsymbol{\psi}_s + (1-q)\boldsymbol{\nu}) \dd s \\& = \frac{1}{2} h(\varepsilon, \varepsilon) + \int_{1/2}^1 h(-\varepsilon, \varepsilon(1/t-1))\dd t\\
& = \frac{1}{2}  h(0,0) +\frac{1}{2}(g_1+g_2) \varepsilon +\frac{1}{2}  h(0,0)  - \frac{1}{2} g_1 \varepsilon + g_2 \varepsilon \int_{1/2}^1  (1/t-1)\dd t + O(\varepsilon^2)\\
& = h(0,0) +  g_2 \varepsilon \log 2  +  O(\varepsilon^2).
\end{align*}
Since we know from \eqref{E:snpd} that $g_2<0$, we can choose $\varepsilon>0$ sufficiently small so that the integral in the right-hand side is strictly less than $h(0,0)$.
\end{proof}

\subsection{Fenchel-Legendre identification of the rate function}
\label{sS:FLid}

Using results of \cite{BM1}, we provide an alternative representation of the rate function $I_q$ as the Fenchel-Lengendre transform of a convex function $\Lambda_q : \R^S \to \R_+$, similar to \eqref{eqn:identifRateFuction}. In order to deal later on with probability measures with support strictly included in $S$, it is convenient to allow some coordinates of $\Lambda_q$ to be equal to $-\infty$; recall the convention \eqref{E:conv}. For every $\lambda \in [-\infty, \infty)^{S}$, we set
\begin{equation}
  \label{E:Lambda}
  \Lambda_q(\lambda)\coloneqq \log q - \log\left(\int_0^{\infty}   \prod (1-t\e^{\lambda(k)})_+^{\nu(k)(1-q) /q}\,  \dd t \right),
\end{equation}
where $x_+$ denotes the positive part of a real number $x$ and the product in the integral runs over $k\in S$ (or, equivalently,  over the $k$'s with $\lambda_k>-\infty$). Needless to say, the right-hand side is interpreted as $-\infty$ when $\lambda\equiv -\infty$.

\begin{theorem}
\label{T:RGE}
We have:
\begin{enumerate}
  \item[(i)] For every $\lambda \in [-\infty, \infty)^{S}$, there is the limit
 $$\lim_{n\to \infty} \frac{1}{n} \log \E_q\left( \exp (\langle L_n, n\lambda\rangle) \right) = \Lambda_q(\lambda).$$
  \item[(ii)] As a consequence, the rate function $I_q$  in Theorem \ref{T:LDPBW} can be identified as the Fenchel-Legendre transform $\Lambda^*_q$ of $\Lambda_q$,
 $$I_q(\boldsymbol{\rho})= \Lambda_q^*(\boldsymbol{\rho})\coloneqq \sup_{\lambda\in \R^{S}}\left( \langle \boldsymbol{\rho}, \lambda\rangle -\Lambda_q(\lambda)\right), \qquad \boldsymbol{\rho}\in \mathcal P_S.$$
\end{enumerate}
\end{theorem}

\begin{proof} (i) {Essentially, this is a weaker (logarithmic)  version of the main result of \cite{BM1}; see also the earlier articles \cite{Fr1, Fr2} in the case $\# S=2$.
Actually, the calculations in \cite{BM1} are mostly done for the special case $\lambda_k=\ln k$, but the method works just as well for arbitrary $\lambda$.
For the reader's convenience, we briefly present below the guiding line from \cite{BM1}, and skip explicit calculations to avoid duplication. }

{To start with, we set $a_k\coloneqq \e^{\lambda_k}$ for $k\in S$, and point at the identities
$$\exp (\langle L_n, n\lambda\rangle)= \prod_{i=1}^{n} a_{\xi_i}=\prod a_k^{\#\{i\leq n: \xi_i=k\}},$$
where the product on the right-hand side runs for $k\in S$.
Next,  introduce  an independent standard Yule process $Y=(Y(t))_{t\geq 0}$ and define  for all $t\geq 0$ and $k\in S$,
$$Y_k(t)\coloneqq \#\{i\leq Y(t): \xi_i=k\}.$$
We argued in \cite[Section 2]{BM1} that $\mathbf{Y}(t)=(Y_k(t))_{k\in S}$ is a multitype Yule process, with space of types $S$.
Since  $Y(t)$ has the geometric distribution with parameter $\e^{-t}$, this yields
\begin{equation}\label{E:pgf}
\e^{-t}\sum_{n=0}^{\infty} (1-\e^{-t})^n\E_q\left( \exp (\langle L_n, n\lambda\rangle)\right)= \E_q\left( \prod a_k^{Y_k(t)}\right);
\end{equation}
see \cite[Lemma 2.1]{BM1}.
 }

{The probability generating function of $\mathbf{Y}(t)$ that appears in the right-hand side of \eqref{E:pgf} has been determined in
\cite[Sections 3 and 4]{BM1} by solving a system of non-linear differential equations; see notably Proposition 4.3 and Corollary 4.4 there. In particular, we can derive the asymptotic behavior of the right-hand side of \eqref{E:pgf} as $t\to \infty$. }

{All what is needed now is to deduce the asymptotic behavior of the moment generating function
$\E_q\left( \exp (\langle L_n, n\lambda\rangle)\right)$ as $n\to \infty$ from that of \eqref{E:pgf} as $t\to \infty$. This is achieved by the singularity analysis performed in \cite[Section 5]{BM1}, under the condition that the maximum of $\lambda$ is reached at a single location.
One gets
$$\lim_{n\to \infty} \exp(-n \Lambda_q(\lambda)) \E_q\left( \exp (\langle L_n, n\lambda\rangle)\right)=c$$
for some explicit constant $c>0$; see  \cite[Theorem 1]{BM1} for the special case $\lambda_k=\ln k$. }

It only remains to consider the situation where the maximum of $\lambda$ is reached at two or more location, say $j_1, \ldots, j_\ell\in S$ with $\ell\geq 2$. But this case immediately reduces to the preceding by considering a quotient space. Specifically set $\tilde S=S\backslash\{j_2, \ldots, j_\ell\}$, and then $\tilde \nu(i)=\nu(i) $ for $i\in \tilde S\backslash \{j_1\}$
and $\tilde \nu(j_1)=\nu(j_1)+ \cdots + \nu(j_\ell)$. Then it should be plain that, in the obvious notation, the distribution of variable
$\langle L_n, n\lambda\rangle$ under $\P_q$ is the same as that of $\langle \tilde L_n, n \tilde\lambda\rangle$ under $\tilde \P_q$, where $\tilde\lambda$ is the restriction of $\lambda$ to $\tilde S$. We conclude that (i) still holds in this situation.

(ii) This follows from (i) by an application of the G\"artner-Ellis Theorem \cite[Theorem 2.3.6]{DZ} and the uniqueness of the rate function, see e.g. \cite[Lemma~4.1.4]{DZ} or \cite[Theorem 1.15]{BD}.
\end{proof}

The  Fenchel-Legendre transform $\Lambda^*_q$   can be evaluated numerically, but unfortunately, it seems that  there is no both general and explicit expression for it (like the variational formula \eqref{E:var} in the case without memory $q=0$),  because the function $\Lambda_q$ in \eqref{E:Lambda}
also cannot be computed explicitly in the first place. There are nonetheless situations where calculations are available, for instance whenever $(1/q-1)\boldsymbol{\nu}$ is integer valued\footnote{ Then $t\mapsto \prod (1-t\e^{\lambda(k)})^{\nu(k)(1-q) /q}$
is a polynom function in the variable $t$, which can thus be integrated explicitly.}.
The example below  illustrates the simplest case of all.

\begin{example} \label{E:es}
Suppose that $S=\{1,2\}$, that the reproduction law is given by $\nu(j)=1/2$ for $j=1,2$. Suppose further that $q=1/3$, so that $(1-q)/q=2$. Then $\Lambda_q(x,y)$ is a symmetric function and
 for any $x\leq y$, we find:
$$\Lambda_{1/3}(x,y) = \log(2) + y -\log\left(3-\e^{x-y}\right).
$$
Thus
$$\nabla \Lambda_{1/3}(x,y)=
\left(\frac{\e^{x-y}}{3-\e^{x-y}},  1-\frac{\e^{x-y}}{3-\e^{x-y}}\right),$$
and we further get for any $0<p\leq 1/2$,
$$\Lambda^*_{1/3}(p,1-p) =\Lambda^*_{1/3}(1-p,p)= p \log\left(\frac{3p}{p+1}\right)-\log(2)+\log\left(\frac{3}{p+1}\right).$$
On the other hand, we have for $\boldsymbol{\rho}=(p,1-p)$ that
$$ H(\boldsymbol{\rho} | q\boldsymbol{\rho} + (1-q) \boldsymbol{\nu})= p \log\left(\frac{3p}{p+1}\right)+(1-p) \log\left(\frac{3(1-p)}{2-p}\right),$$
and the upper-bound \eqref{E:IBW} is remarkably sharp in this case; see Figure \ref{F1}.
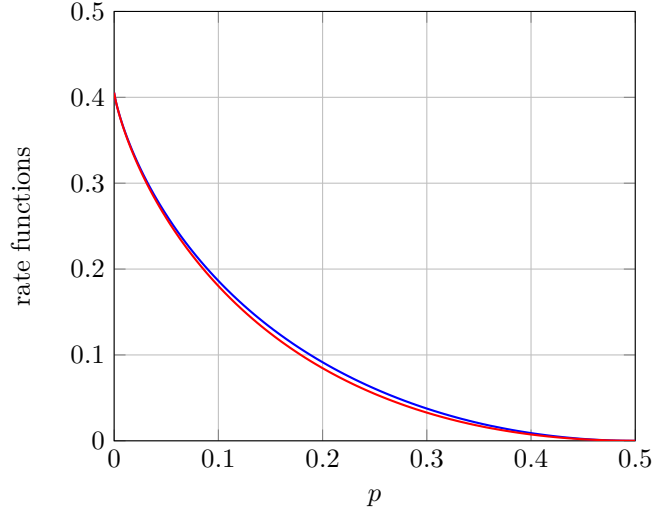
\begin{figure}[!h]
 \begin{center}
\begin{tikzpicture}
    \begin{axis}[
        domain=0:.5,
        samples=500,
        xlabel={$p$},
        ylabel={rate functions},
        grid=both,
        ymin=0, ymax=.5,
        xmin=0, xmax=0.5]
    \addplot[thick,blue] {x*ln(3*x/(x+1)) + (1-x)*ln(3*(1-x)/(2-x))};
    \addplot[thick,red] {x*ln(3*x/(x+1)) - ln(2) + ln(3/(x+1))};
    \end{axis}
\end{tikzpicture}

 \caption{Graphs of  $H(\boldsymbol{\rho} | q\boldsymbol{\rho} + (1-q) \boldsymbol{\nu})$ (\textcolor{blue}{blue}) and of  $\Lambda_q^*(\boldsymbol{\rho})$ (\textcolor{red}{red}) for $\boldsymbol{\rho}=(p,1-p)$ with $p\in[0,1/2]$ in Example \ref{E:es}. The function $H(\cdot | q\cdot + (1-q)\boldsymbol{\nu})$ appears as a very good approximation to $\Lambda_q^*$.}\label{fig:gluingtree}
 \label{F1}
 \end{center}
\end{figure}
\end{example}

\begin{remark}For every $\lambda \in [-\infty, \infty)^{S}$  with $\bar \lambda\coloneqq \sup_{S} \lambda>-\infty$,  the product in  \eqref{E:Lambda} can be bounded from above by $\indset{\{t\leq \exp(-\bar \lambda)\}}$. It follows that
$$\Lambda_q(\lambda) \geq \log q + \bar \lambda,$$
and then,  one has for every $\boldsymbol{\rho}\in \mathcal P_S$,
$$\Lambda^*_q(\boldsymbol{\rho})\leq -\log q.$$
Observe that this agrees with \eqref{E:IBW}, since  obviously $H(\boldsymbol{\rho} | q\boldsymbol{\rho} + (1-q) \boldsymbol{\nu})\leq -\log q$.
\end{remark}

It follows from Theorem \ref{T:RGE} that   $\Lambda_q$ is a convex function on $\R^{S}$, a fact that however does not seem easy to verify directly.
At the opposite, the next first two claims are readily checked directly from the definition  \eqref{E:Lambda}, without appealing to Theorem \ref{T:RGE}.
\begin{lemma} \label{L:Lambda}  The following assertions hold:

\begin{enumerate}
\item[(i)]
For any $\lambda \in [-\infty, \infty)^{S}$ and any $j\in S$ with $\lambda(j)>-\infty$,
  $\Lambda_q$ has a strictly positive partial derivative in the $j$-direction at $\lambda$, $\partial_j \Lambda_q(\lambda)>0$. \\
For $\lambda$ with $\lambda(j)=-\infty$, we further set $\partial_j \Lambda_q(\lambda)=0$.

\item[(ii)] The gradient of $\Lambda_q$,  $\nabla \Lambda_q \coloneqq \left(\partial_j \Lambda_q\right)_{j\in S}$, induces a continuous function
on $\{\lambda \in [-\infty, \infty)^{S}: \lambda_k>-\infty \text{ for some }k\in S\}$ with values in $\mathcal P_S$.
In particular,  the function $\Lambda_q$ is $\mathcal{C}^1$ on $\R^S$.

\item[(iii)] For any $\lambda \in [-\infty, \infty)^{S}$ with $\lambda\not \equiv -\infty$ and $ \boldsymbol{\rho}=\nabla \Lambda_q(\lambda)$,
one has the identity
$$\Lambda_q(\lambda)= \langle \boldsymbol{\rho}, \lambda\rangle-\Lambda^*_q(\boldsymbol{\rho}).
$$
For any $\boldsymbol{\rho}'\in \mathcal P_s$ with $\boldsymbol{\rho}'\neq \nabla \Lambda_q(\lambda)$, one has also
$$\Lambda_q(\lambda)>  \langle \boldsymbol{\rho}', \lambda\rangle-\Lambda^*_q(\boldsymbol{\rho}').
$$
\end{enumerate}
\end{lemma}

\begin{proof} (i-ii) Indeed,  the Leibniz integral rule gives
$$\partial_j \Lambda_q(\lambda)= \nu(j)(1/q-1)\frac{ \int_0^{\infty}  \frac{t\e^{\lambda(j)}}{1-t\e^{\lambda(j)}} \prod (1-t\e^{\lambda(k)})_+^{\nu(k)(1-q) /q}\,  \dd t}{
\int_0^{\infty}   \prod (1-t\e^{\lambda(k)})_+^{\nu(k)(1-q) /q}\,  \dd t}>0.$$
An integration by parts shows that for any $\lambda \in [-\infty, \infty)^{S}$ with $\lambda\not \equiv -\infty$, we have
\begin{align*} & \int_0^{\infty}  \sum_{j\in S} \nu(j)(1/q-1)\frac{t\e^{\lambda(j)}}{1-t\e^{\lambda(j)}} \prod (1-t\e^{\lambda(k)})_+^{\nu(k)(1-q) /q}\,  \dd t \\
=& \int_0^{\infty}   \prod (1-t\e^{\lambda(k)})_+^{\nu(k)(1-q) /q}\,  \dd t,
\end{align*}
which translates into $$\sum_{j\in S}\partial_j\Lambda_q(\lambda)= 1.$$
All the other claims are straightforward.

(iii) Theorem \ref{T:RGE}(i) enables us to apply  \cite[Lemma 2.3.9]{DZ}.
\end{proof}

We deduce some basic properties of the rate function  $\Lambda_q^*$.
\begin{corollary}\label{C:basic} The following assertions hold:
\begin{enumerate}

\item[(i)] The function $\Lambda_q^*$ is strictly convex on $\mathcal P_S$.

\item[(ii)] For every $\boldsymbol{\rho}\in \mathcal P_S$, there exists $\lambda_{\boldsymbol{\rho}} \in [-\infty, \infty)^{S}$  with $\lambda_{\boldsymbol{\rho}}(k)>-\infty$ iff $\rho(k)>0$, such that
$$\Lambda^*_q(\boldsymbol{\rho})= \langle \boldsymbol{\rho}, \lambda_{\boldsymbol{\rho}}\rangle -\Lambda_q(\lambda_{\boldsymbol{\rho}}).$$
Moreover we have $$\boldsymbol{\rho}= \nabla \Lambda_q(\lambda_{\boldsymbol{\rho}}).$$
\end{enumerate}

\end{corollary}
\begin{proof} (i) Indeed, the Fenchel-Legendre transform of a $\mathcal C^1$ function is always strictly convex; see e.g. \cite[Theorem 11.13]{RW}.

(ii) Consider any sequence $(\lambda_n)$ in $\R^S$ such that
$ \langle \boldsymbol{\rho}, \lambda_n\rangle -\Lambda_q(\lambda_n)$ converges to $ \Lambda^*_q(\boldsymbol{\rho})$ as $n\to\infty$. Observe that $\Lambda_q(\lambda+c)=c+\Lambda_q(\lambda)$ for any $\lambda\in \R^S$ and  any $c\in \R$, so there is no loss of generality in assuming that $\bar \lambda_n\coloneqq \max_S \lambda_n=0$ for every $n$. Then, $\lambda_n$  converges  along some subsequence to, say,  $\lambda_{\boldsymbol{\rho}} \in [-\infty, 0]^{S}$  with $\bar \lambda_{\boldsymbol{\rho}}=0$, and this yields the identity
$$\Lambda^*_q(\boldsymbol{\rho})= \langle \boldsymbol{\rho}, \lambda_{\boldsymbol{\rho}}\rangle -\Lambda_q(\lambda_{\boldsymbol{\rho}}).$$
Obviously $\lambda_{\boldsymbol{\rho}}(k)=-\infty$ can only occur when $\rho(k)=0$. In the converse direction, suppose that $\lambda_{\boldsymbol{\rho}}(j)>-\infty$ for some $j\in S$ with $\rho(j)=0$.
Let $\lambda'\in [-\infty, \infty)^{S}$ defined by $\lambda'(k)=\lambda_{\boldsymbol{\rho}}(k)$ if $k\neq j$ and $\lambda'(j)=-\infty$. Then
$\langle \boldsymbol{\rho}, \lambda_{\boldsymbol{\rho}}\rangle= \langle \boldsymbol{\rho}, \lambda'\rangle$
and $\Lambda_q(\lambda_{\boldsymbol{\rho}})> \Lambda_q(\lambda')$, which is absurd. Last, the claim that the gradient of $\Lambda_q$ evaluated at $\lambda_{\boldsymbol{\rho}}$ equals $\boldsymbol{\rho}$ is plain, as
$\lambda_{\boldsymbol{\rho}}$ is a critical point of $\langle \boldsymbol{\rho}, \cdot\rangle- \Lambda_q(\cdot)$.
\end{proof}

The strict convexity of the rate function $\Lambda^*_q$ readily yields the convergence in probability of the empirical distribution conditionally on a large class of rare events, which is a reinforced version of \eqref{E:mostliketraj};  see also e.g \cite[Theorem 1.4]{BD}.

\begin{corollary} \label{C:Rgibbs} Let $\Gamma$ be a  closed convex subspace of $\mathcal P_{S}$ with a non-empty interior. Then there is the convergence in probability
$$ \lim_{n\to \infty}  \P_q(\mathrm{dist}_{\mathcal P_S}(L_{n},\boldsymbol{\gamma}_q)< \varepsilon \mid L_n\in \Gamma) =1, \qquad \text{for all } \varepsilon>0,$$
where $\boldsymbol{\gamma}_q= \arg\min_{\Gamma} \Lambda_q^*$ is the unique location  on $\Gamma$ at which $\Lambda_q^*$ attains its minimum.
\end{corollary}

\begin{remark}
\label{R:nonLine}
{As a consequence of this result and Corollary~\ref{C:IBW}, we observe that contrarily to the usual Sanov's theorem, the path followed by the process of empirical measures, conditioned on being close to $\boldsymbol{\rho}$ at a large time $n$, is not nearly constant. Heuristically, this is due to the fact that it is less costly to influence empirical measures at the start of the process than thereafter.
Informally, a thrifty strategy is thus to drive $L$ during a relatively short time $\floor{\delta n}$ towards a state from which the probability of reaching  a neighborhood of $\boldsymbol{\rho}$ at time $n$ is high.}
\end{remark}

\section{Reinforced Galton--Watson processes} \label{S:Main}

We now state the main results of this paper, which should be viewed as a reinforced version of Theorem \ref{T0}. Recall the notation \eqref{E:rhoell}, as well as that in Theorems \ref{T:LDPBW} and \ref{T:RGE}, and also Lemma \ref{L:Lambda}. Last, note that for $q=0$, the probability measure $\nabla \Lambda_0(\ln)$ coincides with  the size-biased version of the
reproduction law $\bar{\boldsymbol{\nu}}_0$.

\begin{theorem} \label{Tref}  The following claims hold for reinforced Galton--Watson trees with memory parameter $0<q<1$:

\begin{enumerate}
\item[(i)] \textbf{Exponential concentration: }
Define $$\bar{\boldsymbol{\nu}}_q\coloneqq \nabla \Lambda_q(\ln).
$$
Then $\bar{\boldsymbol{\nu}}_q\in \mathcal{P}_S$ and for every neighborhood $G$ of $\bar{\boldsymbol{\nu}}_q$, there exists $\varepsilon=\varepsilon(G)>0$ such that for all $n\geq 1$,
$$\E_q(\#\{v\in T: |v|=n \ \&\  \boldsymbol{\mu}_v\not \in G\})) \leq \e^{-\varepsilon n} \E_q(\#\{v\in T: |v|=n\}).$$

\item[(ii)]  \textbf{Evanescent laws:}
 Any  ${\boldsymbol{\rho}}\in \mathcal P_S$ with
$$\langle \boldsymbol{\rho}, \ln \rangle < \Lambda_q^*(\boldsymbol{\rho}) $$
 is evanescent, $\P_q$-a.s.

 \item[(iii)] \textbf{Strongly persistent laws:}  For any ${\boldsymbol{\rho}}\in \mathcal P_S$ with
$$\langle \boldsymbol{\rho}, \ln \rangle >  H(\boldsymbol{\rho} | q \boldsymbol{\rho} + (1-q) \boldsymbol{\nu} ), $$
the probability under $\P_q$ that $\boldsymbol{\rho}$ is strongly persistent is strictly positive.

\end{enumerate}

 \end{theorem}
\begin{example} As  in Example \ref{E:es}, consider the reproduction law $\boldsymbol{\nu}=\frac{1}{2}\boldsymbol{\delta}_1+ \frac{1}{2}\boldsymbol{\delta}_2$. By Theorem \ref{T0}(i), the empirical offspring distributions of vertices at generation $n\gg 1$ concentrated around the size-biased law $\bar{\boldsymbol{\nu}}_0=\frac{1}{3}\boldsymbol{\delta}_1+ \frac{2}{3}\boldsymbol{\delta}_2$ in the Galton--Watson case, whereas for the reinforced version with memory parameter $q=1/3$, Theorem \ref{Tref}(i) shows that they rather concentrate around $\frac{1}{5}\boldsymbol{\delta}_1+ \frac{4}{5}\boldsymbol{\delta}_2$.
\end{example}

\begin{proof}[Proof of (i)] The proof relies on previous estimates obtained in \cite{BM1}. We shall use these estimates, together with the many-to-one lemma, to compute the average number of individuals at time $n$ with a given empirical distribution.

We know from \cite[Theorem 1.1]{BM1} that
\begin{equation}\label{E:BM1f}
\lim_{n\to \infty} \frac{1}{n} \log \E_q(\#\{v\in T: |v|=n\}) = \Lambda_q(\ln).
\end{equation}
Next,  the many-to-one formula of Lemma \ref{L1} entails that for any $n\geq 1$ and  any measurable $B\subset \mathcal P_S$, there is the
identity
$$\E_q(\#\{v\in T: |v|=n \ \&\  \boldsymbol{\mu}_v\in B\}) = \E_q\left(  \exp(n  \langle \boldsymbol{\mu}_{U_n}, \ln \rangle)
\indset{\boldsymbol{\mu}_{U_n}\in B}\right).$$

The construction of the out-degree sequence  $(d(U_n))$ along a harmonic line under $\P_q$ mirrors that of a reinforced sequence $(\xi_n)$  until the first appearance of the value $0$ (which never occurs a.s. if and only if $\nu(0)=0$, i.e. if and only if there are no empty offsprings). More precisely the former has the same distribution as the latter  stopped at the first index $n$ at which $\xi_n=0$.
Using that $\exp(n\crochet{L_n,\ln}) = 0$ if $0$ appears in the sequence, in terms of the empirical measure, this yields the equality
$$\E_q(\#\{v\in T: |v|=n \ \&\  \boldsymbol{\mu}_v\in B\}) = \E_q\left(  \exp(n  \langle L_n, \ln \rangle )
\indset{L_n\in B}\right).$$

On the other hand, we know from Lemma \ref{L:Lambda}(iii) that
$$  \Lambda_q(\ln) = \langle \bar{\boldsymbol{\nu}}_q, \ln \rangle - \Lambda^*_q(\bar{\boldsymbol{\nu}}_q)$$
and  that
$$ \Lambda_q(\ln) > \langle \boldsymbol{\rho}, \ln \rangle - \Lambda^*_q(\boldsymbol{\rho}) \qquad \text{ for all }\boldsymbol{\rho}\in \mathcal P_S, \ \boldsymbol{\rho}\neq \bar{\boldsymbol{\nu}}_q.$$
By the Laplace principle, we deduce from Theorem  \ref{T:LDPBW} and Theorem \ref{T:RGE}(ii) that
for any closed $F\in \mathcal P_S$ avoiding some neighborhood of $\bar{\boldsymbol{\nu}}_q$,
$$ \limsup_{n\to \infty} \frac{1}{n} \log
\E_q\left(  \exp(n  \langle L_n, \ln \rangle \indset{L_n\in F}\right) \leq \max_{\boldsymbol{\rho}\in F}\left( \langle \boldsymbol{\rho}, \ln \rangle - \Lambda^*_q(\boldsymbol{\rho})
\right) < \Lambda_q(\ln).
$$
Recalling \eqref{E:BM1f}, this established the exponential concentration.
\end{proof}

The argument for (ii) is similar, and actually a bit simpler.

\begin{proof}[Proof of (ii)]  Assume $\langle \boldsymbol{\rho}, \ln \rangle < \Lambda_q^*(\boldsymbol{\rho}) $;
by the lower semi-continuity of $\Lambda^*_q$, we can find a closed neighborhood $F$ of $\boldsymbol{\rho}$ such that
$$\sup_{ \boldsymbol{\rho}'\in F} (\langle \boldsymbol{\rho}', \ln \rangle - \Lambda_q^*(\boldsymbol{\rho}') )=-2c <0.$$
Then, just  as in (i), we have for all $n$ sufficiently large
$$\E_q(\#\{v\in T: |v|=n \ \&\  \boldsymbol{\mu}_v\in F\}) \leq \e^{-cn},$$
and a summation over all generations enables us to conclude that
$$\E_q(\#\{v\in T:  \boldsymbol{\mu}_v\in F\}) <\infty.$$
In particular, $\boldsymbol{\rho}$ is evanescent $\P_q$-a.s.
\end{proof}

 The proof of (iii) relies on the next technical estimate, which can be seen as an extension of results obtained in \cite[Section 5]{BM2}.
 \begin{lemma}\label{L:tech} Consider a function $a: S\to [0, 1/q)$ such that
 $$\sum \frac{ \nu(j)}{1   - qa(j)} = 1/(1-q),$$
 and then define $\boldsymbol{\pi}_a\in \mathcal P_S$ by
$$\pi_a (k) \coloneqq \frac{(1-q) a(k)}{1   - qa(k)} \nu(k),\qquad k\in S.$$
If
$$\sum \pi_a(k) \log \left(\frac{a(k)}{k}\right)<0, $$
then the probability under $\P_q$ that $\boldsymbol{\pi}_a$ is strongly persistent is strictly positive.
 \end{lemma}

 Taking Lemma \ref{L:tech} for granted, we can easily establish the last part of Theorem \ref{Tref}.
 \begin{proof}[Proof of (iii)]  Let ${\boldsymbol{\rho}}\in \mathcal P_S$;
 define
 $$a(k)\coloneqq \frac{\rho(k)}{q \rho(k)+ (1-q)\nu(k)}, \qquad k\in S.$$
 So
 $$H(\boldsymbol{\rho} | q \boldsymbol{\rho} + (1-q) \boldsymbol{\nu} )= \sum \rho(k) \log a(k)
 $$
 and  we have also  $\boldsymbol{\rho}= \boldsymbol{\pi}_a$.
 Therefore the condition of the statement
$$\langle \boldsymbol{\rho}, \ln \rangle >  H(\boldsymbol{\rho} | q \boldsymbol{\rho} + (1-q) \boldsymbol{\nu} )$$
reads
$$\sum \pi_a(k) \log \left(\frac{a(k)}{k}\right) <0.$$
This enables us to invoke Lemma \ref{L:tech}.
 \end{proof}
  The proof of Lemma \ref{L:tech} below may look indirect, however the guiding line is close to that of Theorem \ref{T0}(iii) in the case $q=0$ without reinforcement.
 In short,  under a new probability measure $\P_q^a$, we will construct a spine such that the empirical offspring distribution along the spine approaches $\boldsymbol{\pi}_a$  as the generation goes to infinity. Then, we will have to check that the distribution of the tree under $\P_q^a$ is absolutely continuous with respect to the initial law $\P_q$. This is the role of the inequality in the statement.
This argument is similar to that used in  \cite[Section 5]{BM2}. To minimize the overlap with our previous work, we shall just sketch here the proof of Lemma \ref{L:tech}, focussing on changes needed in key calculations. Of course, reader wishing to follow details may want to have \cite{BM2} at hand as well.

\begin{proof}[Sketch of the proof of Lemma \ref{L:tech}] Let $a$ be as in the statement.
We first introduce
$$m_a(\boldsymbol{\rho}) \coloneqq \sum_{k\in S} a(k) (q\rho(k)+(1-q)\nu(k)), \quad \boldsymbol{\rho}\in \mathcal P_S.$$
This quantity should be interpreted as the conditional mean value of $a(d(v))$  for the reinforced Galton--Watson process, given that $v$ is a vertex with empirical offspring distribution $\boldsymbol{\mu}_v=\boldsymbol{\rho}$.
We further define, for any vertex $v$ of the tree aside the root,
\begin{equation} \label{E:Phi}
\Phi^a({v}) \coloneqq \prod_{j=0}^{|{v}|-1} \frac{a(d(v_j))/d(v_j)}{m_a(\boldsymbol{\mu}_{v_j})},
\end{equation}
where, as previously, $v_j$ denotes the forebear of $v$  at generation $j$, and we agree for definitiveness that for $j=0$, $m_a(\boldsymbol{\mu}_{\varnothing})=\sum  a(k) \nu(k)$.
It is immediate  that the process
\begin{equation}\label{E:martingale}
  M^{a}_n = \sum_{ |{v}|=n} \Phi^a(v), \qquad n\geq 1
\end{equation}
is a nonnegative martingale with unit mean under $\P_q$.

Much as for \eqref{E:defspin}, this martingale serves to define a probability measure $\P_q^a$ for the joint law of a random tree $T$ together with a spine $(V^a_n)_{n\geq 0}$. Specifically, for any $n\geq 1$, any subtree $t_n$ of the Ulam tree with height $n$ and any $v$ vertex of $t_n$ at height $n$, if we write $T_n$ for the subtree obtained by pruning $T$ at generation $n$, then we set
\begin{equation} \label{E:defspinq}
\P^a_q\left(T_n= t_n, V^a_n=v\right)\coloneqq \P_q(T_n=t_n)\Phi^a(v).
\end{equation}
The next key step of the proof lies in the observation that the evolution of the degrees $(d(V^a_n))$ along the spine can be described as a generalized P\'olya urn, whose scheme however differs from that in Section \ref{S:Sanovreinforced}. See \cite[Section 6.2]{BM2}.

Specifically, think again of $S$ as a set of colors; we also
add a special color denoted by $\star$. We define an urn process with balls having colors in $S \cup \{\star\}$ as follows. Let the urn contain two balls at the initial  time $n=0$, one with a random color sampled according to $\boldsymbol{\nu}$ and one with color $\star$.
 Imagine that a ball with color $k\in S$ has activity $qa(k)$, meaning that the probability that it is picked at some random drawing from the urn is proportional to $q a(k)$, whereas a ball with color $\star$ has activity $(1-q) \sum a(j) \nu(j)$.
At each step, a ball is drawn at random in the urn with probability proportional to its activity, its color is observed and the ball is then replaced in the urn. If the color of the sampled ball is  $j \in S$, then we add to the urn one ball with color $j$ and one ball with color $\star$. If the sampled ball has color $\star$, then we add to the urn one ball with color $\star$ and a second ball with random color in $S$, such that  probability of adding the color $k$ is proportional to $a(k)\nu(k)$.
Then the sequence of the successive colors of the non-$\star$ balls added to the urn step after step, has  the same law as the sequence $(d(V^a_n))$ of out-degrees under $\P_q^a$.

The upshot of the description above is that it enables us to determine the asymptotic behavior of $(d(V^a_n))$  using classical results on generalized P\'olya urns due to Athreya and Karlin \cite{AK}, see also Janson \cite{Jan04}. This requires the computation of  the principal spectral elements of  the mean replacement matrix of the urn re-weighted by activities. Namely, introduce $A=(A_{i,j})_{i,j\in S \cup \{\star\}}$ defined for $i,j\in  S$  by
$$
  A_{i,j} = q  a(i)\boldsymbol{\delta}_i(j), \quad A_{i,\star} = q a(i), \quad
   A_{\star,j} = (1-q) a(j)\nu(j), $$
   and finally,
   $$A_{\star,\star}= (1-q)\sum a(k)\nu(k).$$
 One readily finds from our assumption (see \cite[Lemma 5.5]{BM2} for similar calculations) that the leading
 eigenvalue of $A$ is  $1$ and the left-eigenvector can then be chosen to coincide  with $\boldsymbol{\pi}_a$ on $S$.

The first consequence of the above spectral analysis for the P\'olya urn is that the empirical measure of out-degrees along the spine converges $\P^a_q$-a.s. to $\boldsymbol{\pi}_a$
as the generation goes to infinity. On the other hand, one immediately verifies that $m_a(\boldsymbol{\pi}_a)=1$. Thus
$$\lim_{n\to \infty} \frac{1}{n}\log \Phi^a ( V^a_n) = \sum \pi_a(k) \log \left( \frac{a(k)}{k}\right)< 0, \qquad \P^a_q\text{-a.s.},$$
and therefore the series
$\sum_{n=1}^{\infty} \Phi^a ( V^a_n)$ converges $\P^a_q$-a.s. Using the same argument as in the proof of Lemma~\ref{L:spinal}(ii), we conclude that
the martingale $M^{a}_n$ is uniformly integrable under $\P_q$, and it follows that $\boldsymbol{\pi}_a$ is strongly persistent with a strictly positive probability under $\P_q$. \end{proof}

{As an immediate application, we point out that any law on $S$ with no atom at $0$ can be made strongly persistent provided that the reinforcement is sufficiently strong.}

{
\begin{corollary}\label{C:eva-pers}
Let any ${\boldsymbol{\rho}}\in \mathcal P_S$ with  $\rho(0)=0$ and $\rho(1)<1$. We may choose the memory parameter sufficiently large such that
$$\langle \boldsymbol{\rho}, \ln \rangle > (1-q) H(\boldsymbol{\rho} |  \boldsymbol{\nu} ),$$
and then  the probability under $\P_q$ that $\boldsymbol{\rho}$ is strongly persistent is strictly positive.
In particular, if
$$\frac{ \langle \boldsymbol{\rho}, \ln \rangle} {H(\boldsymbol{\rho} |  \boldsymbol{\nu}) } \in (1-q,1),$$
then $\boldsymbol{\rho}$ is evanescent $\P_0$-a.s. and strongly persistent with strictly positive $\P_q$-probability.
\end{corollary}
 }

{\begin{proof}  Indeed $\langle \boldsymbol{\rho}, \ln \rangle >0$ and $H(\boldsymbol{\rho} |  \boldsymbol{\nu} )<\infty$, so we may choose $q\in(0,1)$ as in the statement.
It follows from the convexity of the relative entropy that
$$(1-q) H(\boldsymbol{\rho} |  \boldsymbol{\nu} )\geq  H(\boldsymbol{\rho} |  q\boldsymbol{\rho} +(1-q)\boldsymbol{\nu} ),$$
and we conclude the proof of the first claim with an appeal to Theorem \ref{Tref}(iii). The second assertion follows, by Theorem \ref{T0}(ii) and the variational formula \eqref{E:var}.
\end{proof}
}

The (contraposition of the)  final result of this section result provides an elementary necessary condition for a law to be strongly persistent with positive probability.

\begin{proposition}\label{P:nsp} For any  ${\boldsymbol{\rho}}\in \mathcal P_S$ with
$$qm_{\boldsymbol{\rho}}+ (1-q)m_{\boldsymbol{\nu}} < 1,$$
 the probability  under $\P_q$ that ${\boldsymbol{\rho}}$ being strongly persistent is zero.
\end{proposition}

\begin{proof} Indeed there is then a neighborhood $G$ of ${\boldsymbol{\rho}}$ such that
\begin{equation} \label{E:subcr}
\sup_{\boldsymbol{\rho}'\in G}\left( qm_{\boldsymbol{\rho}'}+ (1-q)m_{\boldsymbol{\nu}} \right) < 1.
\end{equation}
Under $\P_q$, consider any vertex $v$ in $T$ with  $\boldsymbol{\mu}_v\in G$, and prune its descent as soon as the empirical offspring distribution exits $G$.
By subcriticality,  \eqref{E:subcr} entails that the pruned descent of $v$ is finite, $\P_q$-a.s.
\end{proof}

\begin{remark} Comparing Theorem \ref{Tref}(iii) and Proposition \ref{P:nsp} shows that
$$ \langle \boldsymbol{\rho}, \ln \rangle \leq  H( \boldsymbol{\rho}, q \boldsymbol{\rho}+ (1-q)\boldsymbol{\nu}) \text{ as soon as } q m_{\boldsymbol{\rho}}+ (1-q) m_{\boldsymbol{\nu}}=m_{q \boldsymbol{\rho}+ (1-q)\boldsymbol{\nu}} < 1.$$
Actually, it follows immediately from Gibbs inequality that
$$ m_{q \boldsymbol{\rho}+ (1-q)\boldsymbol{\nu}} \leq 1 \Longrightarrow  \langle \boldsymbol{\rho}, \ln \rangle \leq  H( \boldsymbol{\rho}, q \boldsymbol{\rho}+ (1-q)\boldsymbol{\nu}).$$

\end{remark}

\section{A sufficient condition for survival} \label{S:survival}

We now conclude this work by pointing out that Lemma \ref{L:tech} also yields the following  sufficient condition for the survival of reinforced Galton-Watson processes.
Introduce the space of functions on $S$
 $$\mathcal A\coloneqq \left\{ a: S\to [0,1/q) :  \sum \frac{\nu(j)}{1   - qa(j)} = \frac{1}{1-q}\right\},$$
 and the function $J:\mathcal A\to (-\infty, \infty]$ given by
$$ J(a)= \sum \nu(k) \frac{(1-q) a(k)}{1   - qa(k)} \log \left(\frac{a(k)}{k}\right).$$
We also recall that the principal branch $W_0$ of the Lambert  $W$ function is defined for every $x\geq -1/\e$ as the unique solution $y\geq -1$ to the equation $y\e^{y}=x$.
{ \begin{corollary} \label{C1bis}  The reinforced Galton--Watson process survives for ever with a strictly positive $\P_q$-probability whenever
$$\min_{\mathcal A} J<0.$$
More precisely, $J$ always
reaches its minimum at a unique location $\underline a=\arg \min J$, which  is determined as the unique $a\in \mathcal A$
such that there exists some $C>0$ with
$$a(j)= -q^{-1} W_0(-C j),\qquad j\in S.$$
 \end{corollary}}
This is improves upon \cite[Theorem 1.3]{BM2}, in which the sufficient condition for survival is given by the stronger condition $J(b)<0$, where $b\in \mathcal A$ is proportional to the identity, i.e. $b(k)=ck$ for a suitable $c>0$. { In the converse direction, we  believe that the reinforced Galton--Watson process should become eventually extinct a.s. when
$\min_{\mathcal A} J = J(\underline a) >0$.}
\begin{proof} First, recall the notation of  Lemma~\ref{L:tech} and note that
$$J(a) = \sum \pi_a(k) \log \left(\frac{a(k)}{k}\right).$$
So if we assume that $J(\underline a)<0$, then the requirement of Lemma~\ref{L:tech} holds for
 $\boldsymbol{\pi}=\boldsymbol{\pi}_{\underline a}$, and the latter is strongly persistent with a strictly positive $\P_q$-probability. \textit{A fortiori } the reinforced Galton--Watson process survives with a strictly positive probability.

The proof of the second claim uses the method of Lagrange multipliers.
To start with, note that if $0\in S$ and $a\in \mathcal A$ with $a(0)>0$, then $J(a)=\infty$. An argument of compactness and continuity shows that
there exists some $\underline a\in \mathcal{A}$ with $\underline a(0)=0$ whenever $0\in S$, such that $J(\underline a) = \inf J=\min J$.
Next, the function $x\mapsto x\log x$ has a finite derivative at any $x>0$, whereas its right-derivative at $x=0$ is $-\infty$. It follows that if $a\in \mathcal A$ has $a(k)=0$ for some $k\neq 0$,
then $a$ is not a local minimum of $J$; hence $\underline a(k)>0$ for all $k\in S\backslash \{0\}$.

The function $J$ is $\mathcal{C}^1$  on $\{a: S\backslash \{0\} \to (0,1/q)\}$, with a partial derivative  in the $k$-th coordinate
 given by
\[
   \partial_kJ(a) = \frac{(1-q) \nu(k)}{1-q a(k)}\left( 1 + \frac{1}{1-qa(k)} \log\left( \frac{a(k)}{k}\right) \right) .\]
So, if we introduce the constraint function
 $$g(a)=\sum \frac{\nu(j)}{1   - qa(j)},$$
 which has
 $\partial _kg(a)=q\nu(k)/(1-qa(k))^2$,
 then $\mathcal A=\{a: g(a)=1/(1-q)\}$ and the criticality of  $\underline a$
 forces the vectors $\partial_kJ(\underline a)$ and $ \partial _kg(\underline a) $ to be proportional. This entails
the existence of some constant $c\in \R$ such that
$$ \log (\underline {a}(k)/k) + 1 - q \underline{a}(k)= c \qquad \text{for all }k\in S\backslash \{0\}.$$
We rewrite the equation above as
$$ - q \underline a(k) \e^{-q \underline a(k)}=-kC,$$
with $C=\e^c>0$. Thus
$\underline{a}(k) = -q^{-1} W(-C k)$, with $W$ being one of the two branches of Lambert's $W$ function. But since $W_{-1}(-Ck)\leq -1$, only $\underline{a}(j) = -q^{-1} W_0(-C j)$
fits the requirement $qa(k)<1$. Since $W_0$ is monotone increasing, the constraint $g(\underline a)=1/(1-q)$ determines $C$, and we conclude that $J$ has indeed a unique minimum at the location given  in the statement.
 \end{proof}

\paragraph{Funding.}
The first author was partially supported by the project SNSF 200020 212115: Random processes with genealogy dependent dynamics.
The second author was partially supported by the MITI interdisciplinary program 80PRIME GEx-MBB and the ANR MBAP-P (ANR-24-CE40-1833) project.

\bibliographystyle{plain}

\end{document}